\RequirePackage{fix-cm}
\documentclass[smallextended]{svjour3}       
\smartqed  
\usepackage{graphicx}
\usepackage{amsmath,amstext,amssymb,amsfonts}
\usepackage{color}
\newcommand{\ilim} {\mathop{\rm lim\,inf}}

\newcommand{\Xx}{\mathtt{X}}
\newcommand{\Yy}{\mathtt{Y}}
\newcommand{\Aa}{\mathtt{A}}
\newcommand{\Bb}{\mathtt{B}}
\newcommand{\ff}{\mathtt{f}}
\newcommand{\vv}{\mathtt{v}}
\newcommand{\Phii}{\mathtt{\Phi}}
\newcommand{\Psii}{\mathtt{\Psi}}
\newcommand{\K}{\mathbb{K}}
\newcommand{\N}{\mathbb{N}}
\newcommand{\Y}{\mathbb{Y}}

\newcommand{\X}{\mathbb{X}}

\newcommand{\A}{\mathbb{A}}
\newcommand{\R}{\overline{\mathbb{R}}}
\newcommand{\F}{\mathbb{F}}

\newcommand{\B}{\mathcal{B}}

\newcommand{\bb}{\mathbb{B}}
\newcommand{\Gr}{{\rm Gr}}
\newcommand{\kk}{\mathcal{K}}

\begin{document}

\title{Continuity of Equilibria for Two-Person Zero-Sum Games with Noncompact Action Sets and Unbounded Payoffs\thanks{Research of the first author was partially supported by NSF grant CMMI-1636193.} 
}

\titlerunning{Continuity of Equilibria for Two-Person Zero-Sum Games \ldots}        

\author{Eugene~A.~Feinberg \and Pavlo~O.~Kasyanov        \and
        Michael~Z.~Zgurovsky
}


\institute{Eugene~A.~Feinberg \at
              Department of Applied Mathematics and
Statistics,\\
 Stony Brook University,\\
Stony Brook, NY 11794-3600, USA \\ \email{eugene.feinberg@sunysb.edu}
                   \and
           Pavlo~O.~Kasyanov \at
              Institute for Applied System Analysis,\\
National Technical University of Ukraine \\``Igor Sikorsky Kyiv Polytechnic
Institute'',\\ Peremogy ave., 37, build, 35,\\ 03056, Kyiv, Ukraine \\ \email{kasyanov@i.ua}
\and
Michael~Z.~Zgurovsky \at National Technical University of Ukraine \\``Igor Sikorsky Kyiv Polytechnic
Institute'',\\ Peremogy ave., 37, build, 1,\\ 03056, Kyiv, Ukraine \\ \email{zgurovsm@hotmail.com}
}

\date{Received: date / Accepted: date}

\maketitle

\begin{abstract}
This paper extends Berge's maximum theorem for possibly noncompact action sets and unbounded cost functions to minimax problems and studies applications of these extensions to two-player zero-sum games with possibly noncompact action sets and unbounded payoffs.  For games with perfect information, also known under the name of turn-based games,  this paper establishes continuity properties of value functions and solution multifunctions. For games with simultaneous moves,  it provides  results on the existence of lopsided values (the values in the asymmetric form) and solutions. This paper also establishes continuity properties of the lopsided values and solution multifunctions. 
\keywords{Two-person game \and Set-valued mapping \and Continuity of minimax}
\PACS{02.50.Le \and 02.30.Xx \and 02.30.Yy \and 02.30.Sa}
\subclass{91A05 \and 91A44 }
\end{abstract}

\section{Introduction}\label{s3}

Berge's maximum theorem provides sufficient conditions for the continuity of a value function and upper semi-continuity of a solution multifunction.  This theorem plays an important role in control theory, optimization, game theory, and mathematical economics.  The major limitation of the classic Berge's maximum theorem is the assumption that the sets of available controls at each state are compact.  Feinberg et al.~\cite{FKSVAN,FKV,Feinberg et al} generalized Berge's maximum theorem and related results to possibly noncompact sets of actions and introduced the notions of $\K$-inf-compact functions for metric spaces and $\K\N$-inf-compact functions for Hausdorff topological spaces.   These generalizations led to the developments of general optimality conditions for Markov decision processes in Feinberg et al.~\cite{Feinberg et al MDP}, partially observable Markov decision processes in Feinberg et al.~\cite{POMDP}, and inventory control in Feinberg~\cite{Ftut} and Feinberg and Lewis~\cite{FEINBERG2016}; see also Katehakis et al. \cite{Kat1} and Shi et al. \cite{Kat2} for studies of relevant inventory control problems.   The class of $\K$-inf-compact functions is broader than the class of inf-compact functions of two variables.
A function defined on a set of state-action pairs is called $\K$-inf-compact on this set, if this function is inf-compact, when the state variable is restricted to an arbitrary compact subset of the state space; see
Definition~\ref{def:kinf} for details.

This paper studies continuity properties of the value function and solution multifunctions, when a minimax problem is considered  for metric spaces instead of the optimization problem. 
The results are applied to one-step zero-sum games of two players with possibly noncompact action sets and unbounded payoffs.    Section~\ref{s2} presents results relevant to Berge's maximum theorem for noncompact action sets.  Section~\ref{sec:minimax} describes continuity properties of minimax.  In particular, Theorem~\ref{th: BergeMinimax} is the extension of Berge's maximum theorem for metric spaces with possibly noncompact action sets and unbounded costs to the minimax.  Section~\ref{s4} presents results on preserving  $\K$-inf-compactness of a function, when action or state sets are extended to the sets of probability measures on these sets.  Section~\ref{s5} deals with two-person zero-sum games  with possibly noncompact action sets and unbounded payoffs.  The definitions and preliminary facts for games are introduced in Subsection~\ref{sub:prem}.  In particular, the classes of safe and unsafe strategies are introduced, and the lopsided value  (the value in the asymmetric form) is defined.  
Of course, in the case of bounded payoffs,  all the strategies are safe. Theorem~\ref{th:exval} of Subsection~\ref{sub:exval} states the existence of the lopsided value.  Subsection~\ref{sub:sol} introduces sufficient conditions for the existence of solutions for the game.  These  conditions  imply that one of the players players has a compact action set.  This is consistent with the approach undertaken in  Ja\'shkewicz and Nowak~\cite{Jan}, where the most general available results were obtained for stochastic games with compact action sets and unbounded payoffs, and the optimality conditions for one of the players were provided; see also survey~\cite{JanS}.  Subsection~\ref{s5.4} describes continuity properties of the lopsided value,  classic value, and solution multifunctions for the game.  Section~\ref{sec:perf} clarifies that pure strategies are sufficient for games with perfect information, that is, the situation where the second player knows the move of the first player.  
Therefore, the results of Section~\ref{sec:minimax} describe the  properties of solutions for such games. 

The rest of this introduction contains definitions and propositions useful for the understanding of the future material.
Let $\R:=\mathbb{R}\cup\{\pm\infty\}$ and $\mathbb{S}$ be a metric space.
For a nonempty set $S\subset \mathbb{S},$  the notation $f:S\subset\mathbb{S}\mapsto {\R}$ means that for each $s\in S$ the value $f(s)\in{\R}$ is defined.
In general, the function $f$ may be also defined outside of $S.$   The notation $f:\mathbb{S}\mapsto {\R}$ means that the function $f$ is defined on the
entire space $\mathbb{S}.$
This notation is equivalent to the notation $f:\mathbb{S}\subset\mathbb{S}\mapsto {\R}, $ which we do not write explicitly.
For a function $f:S\subset\mathbb{S}\mapsto {\R}$ we sometimes consider its restriction $f\big|_{\tilde{S}}:\tilde{S}\subset\mathbb{S}\mapsto {\R}$ to the set $\tilde{S}\subset S.$
Throughout the paper we denote by $\K(\mathbb{S})$ the
\emph{family of all nonempty compact subsets} of ${\mathbb{S}}$
and by $S(\mathbb{S})$ the \emph{family of all nonempty
subsets} of~$\mathbb{S}.$

We recall that, for a
nonempty set $S\subset \mathbb{S},$ a function $f:S\subset \mathbb{S}\mapsto\R$ is
called \textit{lower semi-continuous at $s\in S$}, if for each sequence
$\{s^{(n)}\}_{n=1,2,\ldots}\subset S,$ that converges to $s$ in $\mathbb{S},$ the
inequality $\ilim_{n\to\infty} f(s^{(n)})\ge f(s)$ holds. A function
$f:S\subset \mathbb{S}\mapsto\R$ is called \textit{upper semi-continuous at
$s\in S$}, if $-f$ is lower semi-continuous at $s\in S.$
Consider the level sets
\[
\mathcal{D}_f(\lambda;S):=\{s\in S \, : \,  f(s)\le \lambda\},\qquad
\lambda\in\mathbb{R}.
\]
The level
sets $\mathcal{D}_f(\lambda;S)$ satisfy the following properties  used in this paper: 

(a) if $\lambda_1>\lambda,$ then $\mathcal{D}_f(\lambda;S)\subset
\mathcal{D}_f(\lambda_1;S);$

(b) if $g,f$ are functions on $S$ satisfying $g(s)\ge f(s)$ for all
$s\in S,$ then $\mathcal{D}_g(\lambda;S)\subset
\mathcal{D}_f(\lambda;S).$

A function
$f:S\subset \mathbb{S}\mapsto\R$ is called \textit{lower / upper semi-continuous}, if $f$ is lower / upper semi-continuous at each $s\in S.$
A function
$f:S\subset \mathbb{S}\mapsto\R$ is called \textit{inf-compact on
$S$}, if all the level sets $\{\mathcal{D}_f(\lambda;S)\}_{\lambda\in\mathbb{R}}$ are compact in $\mathbb{S}.$
A function
$f:S\subset \mathbb{S}\mapsto\R$ is called \textit{sup-compact on
$S$}, if $-f$ is inf-compact on~$S.$

Each nonempty subset $S$
of a metric space $\mathbb{S}$ can be considered as a metric space with the same metric. 
\begin{remark}\label{rem:comp}
{\rm
For each nonempty subset $S\subset\mathbb{S}$ the following equality holds:
\[
\K(S)=\{C\subset S\,:\, C\in \K(\mathbb{S})\}.
\]
}
\end{remark}

\begin{remark}\label{rem:tildelsc}
{\rm
It is well-known that a function $f:\mathbb{S}\mapsto\R$ is lower semi-continuous if and only if
the set $\mathcal{D}_f(\lambda;\mathbb{S})$ is closed for every $\lambda\in\mathbb{R};$ see e.g., Aubin \cite[p.~12, Proposition~1.4]{Aubin}.
For a function $f:S\subset \mathbb{S}\mapsto\R,$ let $\tilde f$ be the function $f:\mathbb{S}\mapsto\R,$ defined as $\tilde{f}(s):=f(s),$ when $s\in S,$ and $\tilde{f}(s):=+\infty$ otherwise. Then the function $\tilde{f}:\mathbb{S}\mapsto\R$ is lower semi-continuous if and only if for each $\lambda\in\mathbb{R}$ the set $\mathcal{D}_f(\lambda;S)$ is closed in $\mathbb{S}.$
}
\end{remark}

Let $\X$ and $\Y$ be metric spaces. For a set-valued mapping $\Phi:\X\mapsto 2^{\Y},$ let
\[
{\rm Dom\,}\Phi:=\{x\in\X\,:\, \Phi(x)\ne \emptyset\}.
\]
A set-valued mapping $\Phi:\X\mapsto 2^{\Y}$ is called \textit{strict} if ${\rm Dom\,}\Phi=\X,$ that is,  $\Phi:\X\mapsto S(\Y)$ or, equivalently,
$\Phi(x)\ne \emptyset$ for each $x\in \X.$
For $Z \subset \X$ define the \textit{graph} of a set-valued mapping
$\Phi:\X\mapsto 2^{\Y},$ restricted to~$Z$:
\[
{\rm Gr}_Z(\Phi)=\{(x,y)\in Z\times\Y\,:\, x\in{\rm Dom\,}\Phi,\, y\in \Phi(x)\}.
\]
When $Z=\X,$ we use the standard notation $\Gr(\Phi)$ for the graph of $\Phi:\X\mapsto 2^{\Y}$ instead of $\Gr_{\X}(\Phi).$

Throughout this section assume that  ${\rm Dom\,}\Phi\ne \emptyset.$ The following definition introduces the notion of a $\K$-inf-compact function defined on $\Gr (\Phi)$ for
$\Phi:\X\mapsto 2^{\Y},$ while in \cite{Feinberg et al} such functions are defined  for $\Phi:\X\mapsto S(\Y).$

\begin{definition}\label{def:kinf} {\rm (cf. Feinberg et al. \cite[Definition 1.1]{Feinberg et al})} A function $f:\Gr(\Phi)\subset\X\times \Y\mapsto \overline{\mathbb{R}}$ is called
$\K$-inf-compact on ${\rm Gr}(\Phi),$
if for every $C\in \K({\rm Dom\,}\Phi)$ this function is inf-compact on ${\rm
Gr}_C(\Phi).$
\end{definition}

\begin{remark}\label{rem:2:9}
{\rm
Each nonempty set $S\subset \X\times\Y$ corresponds the set-valued mapping $\Psi_S:\X\mapsto 2^\Y$ such that
$\Psi_S(x)=\{y\in \Y\,:\, (x,y)\in S\}$ for each $x\in\X.$ We note that ${\rm Dom\,}\Psi_S\ne \emptyset$ and $\Gr(\Psi_S)=S.$ Therefore, when we
write that the function $f:S\subset \X\times\Y\mapsto \R$ is $\K$-inf-compact on $S,$ we mean that $f$ is $\K$-inf-compact on  $\Gr(\Psi_S).$ 
}
\end{remark}

The function $f(x,y)=|x-y|$ is an example of a function $f:\mathbb{R}^2\mapsto\mathbb{R},$ which is $\K$-inf-compact on $\mathbb{R}^2,$ but it is not inf-compact on $\mathbb{R}^2.$  The following example describes another  $\K$-inf-compact function, which is not inf-compact. 

\begin{example}\label{exa:0}
{\rm Let $\X=\Y=\mathbb{R},$ $\Phi(x)=\mathbb{R}$ and $f(x,y)=x+y^2,$ $(x,y)\in \mathbb{R}^2.$ The function $f$ is $\K$-inf-compact on $\mathbb{R}^2$ because  the sets $\mathcal{D}_f(\lambda;C\times\mathbb{R})$ are compact for all $C\in \K(\mathbb{R})$ and $\lambda\in\mathbb{R}$. 
The function $f$ is not inf-compact on $\mathbb{R}^2$ since the level set $ \mathcal{D}_f(0;\mathbb{R}^2)=\{(-y^2,y)\,:\, y\in \mathbb{R}\}$ is not compact.
}  
\end{example}

\begin{definition}\label{def:ksup}
A function $f:\Gr(\Phi)\subset\X\times \Y\mapsto \overline{\mathbb{R}}$ is called
 $\K$-sup-compact on ${\rm Gr}(\Phi)$ if
 the function $-f$ is $\K$-inf-compact on ${\rm
Gr}(\Phi).$
\end{definition}

\begin{remark}\label{rem:Dom}
{\rm
According to Remark~\ref{rem:comp}, a  function $f:\Gr(\Phi)\subset\X\times \Y\mapsto \overline{\mathbb{R}}$ is
$\K$-inf-compact / $\K$-sup-compact on ${\rm Gr}(\Phi)$ if and only if $f:\Gr(\Phi)\subset{\rm Dom\,}\Phi\times \Y\mapsto \overline{\mathbb{R}}$ is
$\K$-inf-compact / $\K$-sup-compact on ${\rm Gr}(\Phi),$ where ${\rm Dom\,}\Phi$ is considered as a metric space with the same metric as on $\X.$
}
\end{remark}

The
topological meaning of $\K$-inf-compactness of a function on a graph of a strict set-valued mapping ${\Phi}:\X\mapsto S(\Y)$ is
explained in Feinberg et al.~\cite[Lemma 2.5]{Feinberg et al}; see also Feinberg et al. \cite[Lemma~2]{FKSVAN} and \cite[p.~1041]{FKV}.
\begin{lemma}{\rm(Feinberg et al.~\cite[Lemma 2.5]{Feinberg et al} and Feinberg and Kasyanov \cite[Lemma~2]{FKSVAN})}\label{k-inf-compact-strict}
Let ${\Phi}:\X\mapsto S(\Y)$ be a strict set-valued mapping. Then the function $f:\Gr(\Phi)\subset\X\times \Y\mapsto \overline{\mathbb{R}}$ is
$\K$-inf-compact on ${\rm Gr}(\Phi)$ if and only if the following two assumptions hold:
\begin{itemize}
\item[{\rm
(i)}] for each $\lambda\in\mathbb{R}$ the set $\mathcal{D}_f(\lambda;\Gr(\Phi))$ is closed in $\X\times\Y;$
\item[{\rm(ii)}] if a sequence $\{x^{(n)} \}_{n=1,2,\ldots}$ with values in $\X$
converges and its limit $x$ belongs to $\X,$ then each sequence $\{y^{(n)}
\}_{n=1,2,\ldots}$ with $y^{(n)}\in \Phi(x^{(n)}),$ $n=1,2,\ldots,$ satisfying
the condition that the sequence $\{f(x^{(n)},y^{(n)}) \}_{n=1,2,\ldots}$ is
bounded above, has a limit point $y\in \Phi(x).$
\end{itemize}
\end{lemma}

The following lemma provides necessary and sufficient conditions for $\K$-inf-compactness
of a function $f:\Gr(\Phi)\subset\X\times \Y\mapsto \overline{\mathbb{R}}$ for a possibly non-strict set-valued mapping ${\Phi}:\X\mapsto 2^{\Y}.$
\begin{lemma}\label{k-inf-compact}
The function $f:\Gr(\Phi)\subset\X\times \Y\mapsto \overline{\mathbb{R}}$ is
$\K$-inf-compact on ${\rm Gr}(\Phi)$ if and only if the following two assumptions hold:
\begin{itemize}
\item[{\rm
(i)}] $f:\Gr(\Phi)\subset\X\times \Y\mapsto \overline{\mathbb{R}}$ is lower semi-continuous;
\item[{\rm(ii)}] if a sequence $\{x^{(n)} \}_{n=1,2,\ldots}$ with values in ${\rm Dom\,}\Phi$
converges in $\X$ and its limit $x$ belongs to ${\rm Dom\,}\Phi,$ then each sequence $\{y^{(n)}
\}_{n=1,2,\ldots}$ with $y^{(n)}\in \Phi(x^{(n)}),$ $n=1,2,\ldots,$ satisfying
the condition that the sequence\\ $\{f(x^{(n)},y^{(n)}) \}_{n=1,2,\ldots}$ is
bounded above, has a limit point $y\in \Phi(x).$
\end{itemize}
\end{lemma}
\begin{proof}
According to Remark~\ref{rem:Dom}, the function $f:\Gr(\Phi)\subset\X\times \Y\mapsto \overline{\mathbb{R}}$ is
$\K$-inf-compact on ${\rm Gr}(\Phi)$ if and only if the function $f:\Gr(\Phi)\subset{\rm Dom}\Phi\times \Y\mapsto \overline{\mathbb{R}}$ is
$\K$-inf-compact on ${\rm Gr}(\Phi),$  where ${\rm Dom\,}\Phi$ is considered as a metric space with the same metric as on $\X.$ Therefore, Lemma~\ref{k-inf-compact-strict}, being applied to
$\X={\rm Dom\,}\Phi,$ $\Y=\Y,$ $f=f,$ and $\Phi=\Phi\big|_{{\rm Dom\,}\Phi},$ implies that the function $f:\Gr(\Phi)\subset\X\times \Y\mapsto \overline{\mathbb{R}}$ is
$\K$-inf-compact on ${\rm Gr}(\Phi)$ if and only if the following two assumptions hold:
\begin{itemize}
\item[{\rm
(a)}] for each $\lambda\in\mathbb{R}$ the set $\mathcal{D}_f(\lambda;\Gr(\Phi))$ is closed in ${\rm Dom\,}\Phi\times\Y;$
\item[{\rm(b)}] assumption (ii) of Lemma~\ref{k-inf-compact} holds.
\end{itemize}
The rest of the proof establishes that, under assumption (b), assumption~(a) holds if and only if  assumption (i) of Lemma~\ref{k-inf-compact} holds.

Let us prove that assumptions (a) and (b) imply assumption (i) of Lemma~\ref{k-inf-compact}. Consider a sequence $\{(x^{(n)},y^{(n)})\}_{n=1,2,\ldots}\subset \Gr(\Phi)$ that converges to $(x,y)\in \Gr(\Phi).$ Then either $\ilim_{n\to\infty}
f(x^{(n)},y^{(n)})=+\infty$ or there exists a subsequence $\{(x^{(n_k)},y^{(n_k)})\}_{k=1,2,\ldots}\subset
\{(x^{(n)},y^{(n)})\}_{n=1,2,\ldots}$ such that, for each real
$\lambda > \ilim_{n\to\infty} f(x^{(n)},y^{(n)}),$ the sequence $\{(x^{(n_k)},y^{(n_k)})\}_{k=1,2,\ldots}$ is eventually in $\mathcal{D}_f(\lambda;\Gr(\Phi)).$  Since the set $\mathcal{D}_f(\lambda;\Gr(\Phi))$ is closed in ${\rm Dom\,}\Phi\times\Y,$ we have $(x,y)\in \mathcal{D}_f(\lambda;\Gr(\Phi))$ for each real $\lambda > \ilim_{n\to\infty} f(x^{(n)},y^{(n)})$ and, therefore, \[f(x,y)\le \ilim_{n\to\infty} f(x^{(n)},y^{(n)}),\] that is, assumption (i) of Lemma~\ref{k-inf-compact} holds. 

Let assumption (b) and assumption (i) of Lemma~\ref{k-inf-compact} hold. Then (a) holds. Indeed, we fix an arbitrary $\lambda \in\mathbb{R}$ and prove that the level set $\mathcal{D}_f(\lambda;\Gr(\Phi))$ is closed in ${\rm Dom\,}\Phi\times\Y.$ 
Let $\{(x^{(n)},y^{(n)})\}_{n=1,2,\ldots}\subset \mathcal{D}_f(\lambda;\Gr(\Phi))$ be a sequence that converges and its limit $(x,y)$ belongs to ${\rm Dom\,}\Phi\times\Y.$
Assumption~(b) implies that $(x,y)\in\Gr(\Phi).$ Moreover, since $f:\Gr(\Phi)\subset\X\times \Y\mapsto \overline{\mathbb{R}}$ is lower semi-continuous, this function is lower semi-continuous at $(x,y)\in \Gr(\Phi).$ Therefore, 
\[
f(x,y)\le \ilim_{n\to\infty}f(x^{(n)},y^{(n)})\le \lambda,
\]
that is, $(x,y)\in \mathcal{D}_f(\lambda;\Gr(\Phi)).$ Thus  
the set $\mathcal{D}_f(\lambda;\Gr(\Phi))$ is closed in ${\rm Dom\,}\Phi\times\Y$
for arbitrary $\lambda\in\mathbb{R}.$ Assumption (a) holds.
\qed \end{proof}

The following corollary establishes that assumption~(i) in Lemma~\ref{k-inf-compact-strict} can be substituted by lower semi-continuity of $f:\Gr(\Phi)\subset\X\times \Y\mapsto \overline{\mathbb{R}}.$

\begin{corollary}\label{strict lsc on}
Let ${\Phi}:\X\mapsto S(\Y)$ be a strict set-valued mapping and $f:\Gr(\Phi)\subset\X\times \Y\mapsto \overline{\mathbb{R}}$ be a function satisfying assumption~(ii) of Lemma~\ref{k-inf-compact-strict}. Then for each $\lambda\in\mathbb{R}$ the set $\mathcal{D}_f(\lambda;\Gr(\Phi))$ is closed in $\X\times\Y$ if and only if the function $f:\Gr(\Phi)\subset\X\times \Y\mapsto \overline{\mathbb{R}}$ is lower semi-continuous.
\end{corollary}
\begin{proof}
This corollary follows directly from Lemmas~\ref{k-inf-compact-strict} and \ref{k-inf-compact}.  
\qed \end{proof}

%

A set-valued mapping ${F}:\X
\mapsto 2^{\Y}$ is \textit{upper semi-continuous} at $x\in{\rm Dom\,}F$ if, for
each neighborhood $\mathcal{G}$ of the set $F(x),$ there is a
neighborhood of $x,$ say $U(x),$ such that
$F(x^*)\subset \mathcal{G}$ for all $x^*\in U(x)\cap {\rm Dom\,}F;$ a
set-valued mapping ${F}:\X
\mapsto 2^{\Y}$ is \textit{lower
semi-continuous} at $x\in{\rm Dom\,}F$ if, for each open set $\mathcal{G}$
with $F(x) \cap \mathcal{G} \neq \emptyset,$ there is a
neighborhood of $x,$ say $U(x),$ such that if $x^*\in
U(x)\cap {\rm Dom\,}F ,$ then
$F(x^*)\cap \mathcal{G}\ne\emptyset$ (see e.g., Berge
\cite[p.~109]{Ber} or Zgurovsky et al. \cite[Chapter~1,
p.~7]{ZMK1}). We note that a
set-valued mapping ${F}:\X
\mapsto 2^{\Y}$ is {lower
semi-continuous} at $x\in{\rm Dom\,}F$ if and only if, 
for each sequence $\{x^{(n)}\}_{n=1,2,\ldots}\subset {\rm Dom\,}F$ converging to 
$x$ and for each $y\in F(x),$ there exists a sequence $\{y^{(n)}\}_{n=1,2,\ldots}$  such that $y^{(n)}\in F(x^{(n)})$ and $y$ is a limit point of $\{y^{(n)}\}_{n=1,2,\ldots}.$  A set-valued mapping is called \textit{upper / lower
semi-continuous}, if it is upper /
lower
semi-continuous at all $x\in{\rm Dom\,}F.$

The following sufficient conditions for $\K$-inf-compactness were introduced in Feinberg et al.~\cite[Lemma~2.1]{Feinberg et al} for $\Phi:\X\mapsto S(\Y).$ 
\begin{lemma}\label{lm0}
Let $\Phi:\X\mapsto 2^{\Y}$ be a set-valued mapping and  $f:\Gr(\Phi)\subset\X\times \Y\mapsto
\overline{\mathbb{R}}$ be a function. Then the following statements hold:
\begin{itemize}
\item[{\rm(a)}] if $f:{\rm Gr}(\Phi)\subset\X\times \Y\mapsto \overline{\mathbb{R}}$ is inf-compact on
${\rm Gr}(\Phi),$ then the function $f$ is
$\K$-inf-compact on ${\rm Gr}(\Phi);$
\item[{\rm(b)}] if $f:{\rm Gr}(\Phi)\subset\X\times \Y\mapsto \overline{\mathbb{R}}$ is lower
semi-continuous and $\Phi:\X\mapsto 2^{\Y}$ is upper semi-continuous and compact-valued at each $x\in{\rm Dom\,}\Phi,$ then the function
$f$ is $\K$-inf-compact on ${\rm Gr} ({{\Phi}}).$
\end{itemize}
\end{lemma}
\begin{proof}
In view of Remark~\ref{rem:comp}, Feinberg et al.~\cite[Lemma~2.1]{Feinberg et al}, being applied to
$\Xx:={\rm Dom\,}\Phi,$ $\Yy:=\A,$ $u:=f,$ and $\Phii:=\Phi\big|_{\Xx},$ implies all the statements of the lemma.
\qed \end{proof}
\begin{definition}\label{def:kuppersemicomp}
{\rm (cf. Feinberg et al. \cite[Definition~2.3]{FKV})}
A set-valued mapping $F:\X\mapsto 2^{\Y}$ is \textit{$\K$-upper semi-compact} if for each $C\in\K({\rm Dom\,}F)$ the set ${\rm Gr}_C(F)$
is compact.
\end{definition}

The following lemma provides the necessary and sufficient conditions for $\K$-upper semi-compactness of
a possibly non-strict set-valued mapping $\Phi:\X\mapsto 2^{\Y}.$ For $\Phi:\X\mapsto S(\Y),$ this statement follows from Feinberg et al.~\cite[Theorem~2.5]{FKV}.
\begin{lemma}\label{k-upper semi-comp}
A set-valued mapping $\Phi:\X\mapsto 2^{\Y}$ is $\K$-upper semi-compact if and only if it is upper semi-continuous and compact-valued at each $x\in{\rm Dom\,}\Phi.$
\end{lemma}
\begin{proof}
In view of Remark~\ref{rem:comp}, Feinberg et al.~\cite[Theorem~2.5]{FKV}, being applied to 
$\Xx:={\rm Dom\,}\Phi,$ $\Yy:=\A,$ $u:=f,$ and $\Psi:=\Phi\big|_{\Xx},$ 
 implies the statement of the lemma.
\qed \end{proof}

\section{Continuity Properties of Minima}\label{s2}
Let $\X, \Y$ be metric spaces, $\Phi:\X\mapsto 2^{\Y}$ be a set-valued mapping with ${\rm Dom\,}\Phi\ne \emptyset,$ and  $f:\Gr(\Phi)\subset\X\times \Y\mapsto 
\overline{\mathbb{R}}$ be a function.  Define
the \textit{value function}
\begin{equation}\label{eq1star}
f^*(x):=\inf\limits_{y\in {\Phi}(x)}f(x,y),\qquad 
x\in{\rm Dom\,}\Phi,
\end{equation}
and the \textit{solution multifunction}
\begin{equation}\label{e:defFi*}
{\Phi}^*(x):=\left\{y\in
{\Phi}(x):\,f^*(x)=f(x,y)\right\},\quad x\in{\rm Dom\,}\Phi.
\end{equation}
%

According to Berge's theorem  \cite[Theorem~2, p.~116]{Ber}, under assumptions of Lemma~\ref{lm0}(b), the function $f^*$ is lower semi-continuous if the set-valued mapping $\Phi:\X\mapsto2^{\Y}$ is strict.  For metric spaces $\X$ and $\Y,$
the following theorem generalizes Berge's theorems from Feinberg et al. \cite[Theorems~2.1(ii) and 3.4]{FKV} and \cite[Theorem~3.1]{Feinberg et al} to a possibly non-strict set-valued mapping $\Phi:\X\mapsto2^{\Y}.$
\begin{theorem}\label{th:BGL}
If a function $f:{\rm Gr}(\Phi)\subset \X \times \Y \mapsto \overline{\mathbb{R}}$ is
$\K$-inf-compact on $\Gr(\Phi),$ then the value function $f^*:{\rm Dom\,}\Phi \subset
\X\mapsto \overline{\mathbb{R}}$ defined in (\ref{eq1star}) is lower semi-continuous.
Moreover, the infimum in (\ref{eq1star}) can be
replaced with the minimum and the nonempty sets $\{{\Phi}^*(x)\}_{x\in{\rm Dom\,}\Phi}$ defined in (\ref{e:defFi*})
satisfy the following properties:
\begin{itemize}
\item[{\rm(a)}] the graph ${\rm Gr}({\Phi}^*)$ is a Borel subset of\, $\X\times \Y;$
\item[{\rm(b)}] if $f^*(x)=+\infty,$
then ${\Phi}^*(x)={\Phi}(x),$ and, if $f^*(x)<+\infty,$ then
${\Phi}^*(x)$ is compact; $x\in {\rm Dom\,}\Phi.$
\end{itemize}
\end{theorem}
\begin{proof}
According to Remark~\ref{rem:comp}, Feinberg et al. \cite[Theorems~2.1(ii) and 3.4]{FKV}, being applied to
$\Xx:={\rm Dom\,}\Phi,$ $\Yy:=\A,$ $u:=f,$ and $\Phii:=\Phi\big|_{\Xx},$ implies that the value function $f^*:{\rm Dom\,}\Phi \subset
\X\mapsto \overline{\mathbb{R}}$ is lower semi-continuous.
Moreover, Feinberg et al. \cite[Theorem~3.1]{Feinberg et al}, being applied to $\Xx:={\rm Dom\,}\Phi,$ $\Yy:=\A,$ $u:=f,$ and $\Phii:=\Phi\big|_{\Xx},$ implies that the infimum in (\ref{eq1star}) can be
replaced with the minimum and the nonempty sets $\{{\Phi}^*(x)\}_{x\in{\rm Dom\,}\Phi}$ defined in (\ref{e:defFi*})
satisfy properties (a) and (b).
\qed \end{proof}

The following theorem describes sufficient conditions for upper semi-continuity of the value function $f^*$ defined in (\ref{eq1star}). A more general result is presented in Feinberg and Kasyanov \cite[Theorem~4]{FKSVAN}, which can be   generalized to a possibly nonstrict set-valued mapping $\Phi.$ However, for the purposes of this paper we need only the following theorem for metric spaces.
\begin{theorem}{\rm (Hu and Papageorgiou \cite[Proposition~3.1, p.~82]{Hu})}\label{th:BGU}
If a set-valued mapping $\Phi: \X \mapsto S({\Y})$ is lower
semi-continuous and a function $f:{\rm Gr}(\Phi)\subset \X \times \Y \mapsto \overline{\mathbb{R}}$ is
upper semi-continuous, then the value function $f^*:
\X\mapsto \overline{\mathbb{R}}$ defined in (\ref{eq1star}) is upper semi-continuous.
\end{theorem}

The following theorem describes sufficient conditions
for $\K$-upper semi-compactness of the solution multifunction $\Phi^*$ defined in (\ref{e:defFi*}); see also Lemma~\ref{k-upper semi-comp}.

\begin{theorem} {\rm (Feinberg and Kasyanov \cite[Theorem~5]{FKSVAN} and Feinberg et al. \cite[p.~1045]{FKV})} \label{th:usc}
Let $\Phi: \X \mapsto S({\Y}),$ a function $f:{\rm Gr}(\Phi)\subset \X\times\Y\mapsto \overline{\mathbb{R}}$ be  $\K$-inf-compact on $\Gr(\Phi),$ and the value function
$f^*:\X\mapsto \mathbb{R}\cup\{-\infty\}$ defined in (\ref{eq1star}) be continuous. Then the infimum in (\ref{eq1star}) can be
replaced with the minimum and the solution multifunction $\Phi^*:\X\mapsto S(\Y)$ defined in (\ref{e:defFi*}) is $\K$-upper
semi-compact.
\end{theorem}

%

\section{Continuity Properties of Minimax}\label{sec:minimax}
This section describes continuity properties of minimax and solution multifunctions. These results are applied in Subsection~\ref{s5.4}, where continuity properties of the lopsided value, classic value, and solution multifunctions for the two-person zero-sum games  with possibly noncompact action sets and unbounded payoffs are described. For metric spaces the presented results can be viewed as extensions of Berge's maximum theorem for noncompact image sets and relevant statements for optimization problems from Feinberg et al.~\cite{FKV,Feinberg et al} to minimax settings.

The minimax problem introduced and studied in this section models robust optimization problems and two-person zero-sum one-step games with perfect information. In such games, players make decisions sequentially, and these games are called sometimes turn-based.  Unlike the case of games with simultaneous moves studied in Section~\ref{s5}, pure policies are sufficient for games with perfect information, and this is formally explained in Section~\ref{sec:perf}. 

Let $\Xx, \Aa$ and $\Bb$ be metric spaces, $\Phii_{\Aa}:\Xx\mapsto S(\Aa)$ and $\Phii_{\Bb}:\Gr(\Phii_{\Aa})\subset\Xx\times\Aa\mapsto S(\Bb)$ be set-valued mappings, and $\ff:{\rm Gr}(\Phii_{\Bb})\subset \Xx \times \Aa\times\Bb \mapsto \overline{\mathbb{R}}$ be a function. Define
the \textit{worst-loss function}
\begin{equation}\label{eq1starworstloss}
\ff^\sharp(x,a):=\sup\limits_{b\in \Phii_{\Bb}(x,a)}\ff(x,a,b),\qquad
(x,a)\in\Gr(\Phii_{\Aa}),
\end{equation}
the \textit{minimax} or \textit{upper value function}
\begin{equation}\label{eq1starminimax}
\vv^\sharp(x):=\inf\limits_{a\in \Phii_{\Aa}(x)}\sup\limits_{b\in \Phii_{\Bb}(x,a)}\ff(x,a,b),\qquad
x\in\Xx,
\end{equation}
and the \textit{solution multifunctions}
\begin{equation}\label{e:defFi*minimax1}
\Phii_{\Aa}^*(x):=\big\{a\in
\Phii_{\Aa}(x)\,:\,\vv^\sharp(x)=\sup\limits_{b\in \Phii_{\Bb}(x,a)}\ff(x,a,b)\big\},\quad x\in\Xx;
\end{equation}
\begin{equation}\label{e:defFi*minimax2}
\Phii_{\Bb}^*(x,a):=\big\{b\in
\Phii_{\Bb}(x,a)\,:\,\sup\limits_{b^*\in \Phii_{\Bb}(x,a)}\ff(x,a,b^*)=\ff(x,a,b)\big\},\ (x,a)\in\Gr(\Phii_{\Aa}).
\end{equation}
We note that the following equalities hold:
\begin{equation}\label{eq:equal}
\begin{aligned}
\vv^\sharp(x) =\inf\limits_{a\in \Phii_{\Aa}(x)}\ff^\sharp(x,a),\quad \Phii_{\Aa}^*(x) =\big\{a\in
\Phii_{\Aa}(x)\,:\,\vv^\sharp(x)=\ff^\sharp(x,a)\big\}, \quad x\in\X;\\
\Phii_{\Bb}^*(x,a) =\big\{b\in
\Phii_{\Bb}(x,a)\,:\,\ff^\sharp(x,a)=\ff(x,a,b)\big\}, \quad(x,a)\in\Gr(\Phii_{\Aa}).
\end{aligned}
\end{equation}
%

The rest of this section establishes sufficient conditions for: 
\begin{itemize}
\item[(i)] continuity properties of the worst-loss function $\ff^\sharp$ (Theorems~\ref{cor:wloss_lsc}, \ref{th:wloss_kinfcomp}, \ref{th:wloss_usc},
\ref{th:wloss_cont},  and~\ref{th: BergeMinimax}),
\item[(ii)] continuity properties of the minimax function $\vv^\sharp$ (Theorems~\ref{th:minimax_lsc}, \ref{th:minimax_usc}, \ref{th:minimax_cont}, and \ref{th: BergeMinimax}),
\item[(iii)] continuity properties of the solution multifunctions $\Phii_{\Aa}^*$ and $\Phii_{\Bb}^*$ (Theorems~\ref{th: uscphii1}, \ref{th: uscphii2}, and \ref{th: BergeMinimax}),
\end{itemize}
when the image sets $\{\Phii_{\Aa}(x)\}_{x\in\Xx}$ and $\{\Phii_{\Bb}(x,a)\}_{(x,a)\in \Gr(\Phii_{\Aa})}$ can be noncompact. 

To state the main results of this section, we introduce
the set-valued mapping  $\Phii_\Bb^{\Aa\leftrightarrow \Bb}:\Xx\times\Bb\mapsto 2^\Aa$ uniquely defined by its graph,
\begin{equation}\label{eq:auxil23}
\Gr(\Phii_\Bb^{\Aa\leftrightarrow \Bb}):=\{(x,b,a)\in \Xx\times\Bb\times\Aa\,:\, (x,a,b)\in \Gr(\Phii_{\Bb})\},
\end{equation}
that is,
\[
\Phii_\Bb^{\Aa\leftrightarrow \Bb}(x,b)=\{a\in \Phii_\Aa(x)\,:\, b\in \Phii_\Bb(x,a)\},
\]
$(x,b)\in {\rm Dom\,}\Phii_\Bb^{\Aa\leftrightarrow \Bb}.$
We also introduce
the function $\ff^{\Aa\leftrightarrow \Bb}:\Gr(\Phii_\Bb^{\Aa\leftrightarrow \Bb})\subset (\Xx\times\Bb)\times\Aa\mapsto\R,$
\begin{equation}\label{eq:auxil3}
\ff^{\Aa\leftrightarrow \Bb}(x,b,a):=\ff(x,a,b),\quad (x,a,b)\in \Gr(\Phii_{\Bb}).
\end{equation}
%
According to (\ref{eq:auxil23}), the following equalities hold:
\begin{equation}\label{eq:auxil4}
\begin{aligned}
{\rm Dom\,}\Phii_\Bb^{\Aa\leftrightarrow \Bb}={\rm proj}_{\Xx\times\Bb}\Gr&(\Phii_{\Bb})=\{(x,b)\in \Xx\times\Bb\,:\,\\
&(x,a,b)\in  \Gr(\Phii_{\Bb}){\rm \ for\ some\ }a\in \Aa\},
\end{aligned}
\end{equation}
where ${\rm proj}_{\Xx\times\Bb}\Gr(\Phii_{\Bb})$ is a projection  of $\Gr(\Phii_{\Bb})$ on
$\Xx\times\Bb.$
\begin{remark}\label{rem:substit}
{\rm
According to Lemma~\ref{k-inf-compact}, the function $\ff^{\Aa\leftrightarrow \Bb}:\Gr(\Phii_\Bb^{\Aa\leftrightarrow \Bb})\subset (\Xx\times\Bb)\times\Aa\mapsto\R$ defined in (\ref{eq:auxil3}), where $\Phii_\Bb^{\Aa\leftrightarrow \Bb}$ is defined in (\ref{eq:auxil23}), is $\K$-inf-compact on $\Gr(\Phii_\Bb^{\Aa\leftrightarrow \Bb})$ if and only if the following two conditions hold:
\begin{itemize}
\item[(i)] the function $\ff:{\rm Gr}(\Phii_{\Bb})\subset \Xx \times \Aa\times\Bb \mapsto \overline{\mathbb{R}}$ is lower semi-continuous;
\item[(ii)] if a sequence $\{x^{(n)},b^{(n)}\}_{n=1,2,\ldots}$ with values in ${\rm Dom\,}\Phii_\Bb^{\Aa\leftrightarrow \Bb}$
converges and its limit $(x,b)$ belongs to ${\rm Dom\,}\Phii_\Bb^{\Aa\leftrightarrow \Bb},$ then each sequence $\{a^{(n)}
\}_{n=1,2,\ldots}$ with $(x^{(n)},a^{(n)},b^{(n)})\in {\rm Gr}(\Phii_{\Bb}),$ $n=1,2,\ldots,$ satisfying
the condition that the sequence $\{\ff(x^{(n)},a^{(n)},b^{(n)}) \}_{n=1,2,\ldots}$ is
bounded above, has a limit point $a\in \Phii_\Bb^{\Aa\leftrightarrow \Bb}(x,b).$
\end{itemize}
}
\end{remark}
The following theorem establishes  sufficient conditions for lower semi-continuity of the worst-loss function $\ff^\sharp:\Gr(\Phii_{\Aa})\subset
\Xx\times\Aa\mapsto \overline{\mathbb{R}}$ defined in (\ref{eq1starworstloss}), when the image sets $\{\Phii_{\Aa}(x)\}_{x\in\Xx}$ and $\{\Phii_{\Bb}(x,a)\}_{(x,a)\in \Gr(\Phii_{\Aa})}$ are possibly noncompact.

\begin{theorem}{\rm(Lower semi-continuity of the worst-loss function)}\label{cor:wloss_lsc}
Let $\Phii_{\Bb}:\Gr(\Phii_{\Aa})\subset\Xx\times\Aa\mapsto S(\Bb)$ be a lower semi-continuous set-valued mapping and the function $\ff:{\rm Gr}(\Phii_{\Bb})\subset \Xx \times \Aa\times\Bb \mapsto \overline{\mathbb{R}}$ be lower semi-continuous. Then the worst-loss function $\ff^\sharp:
\Gr(\Phii_{\Aa})\subset\Xx\times\Aa\mapsto \overline{\mathbb{R}}$ defined in (\ref{eq1starworstloss}) is lower semi-continuous.
\end{theorem}
\begin{proof}
Theorem~\ref{th:BGU},  applied to $\X:=\Gr(\Phii_{\Aa}),$ $\Y:=\Bb,$ $\Phi:=\Phii_{\Bb},$ and $f:=-\ff,$ implies that the function $\ff^\sharp:  
\Gr(\Phii_{\Aa})\subset\Xx\times\Aa\mapsto \overline{\mathbb{R}}$ is
 lower semi-continuous.
%
%
\qed \end{proof}

To state sufficient conditions for the $\K$-inf-compactness of the worst-loss function (see Theorem~\ref{th:wloss_kinfcomp}), we need to introduce the   $\Aa$-lower semi-continuity assumption
for a set-valued mapping $\Phii_{\Bb}:\Gr(\Phii_{\Aa})\subset\Xx\times\Aa\mapsto S(\Bb),$ which implies its lower semi-continuity.

\begin{definition}\label{defi:uniformAlsc}
A set-valued mapping $\Phii_{\Bb}:\Gr(\Phii_{\Aa})\subset\Xx\times\Aa\mapsto S(\Bb)$ is called {\it 
$\Aa$-lower semi-continuous}, if the following condition holds:
\begin{itemize}
\item[]
if a sequence $\{x^{(n)}\}_{n=1,2,\ldots}$ with values in $\Xx$ converges and its limit $x$ belongs to $\Xx,$ $a^{(n)}\in \Phii_\Aa(x^{(n)})$ for each $n=1,2,\ldots,$ and $b\in \Phii_\Bb(x,a)$ for some $a\in \Phii_\Aa(x),$ then there is a sequence $\{b^{(n)}\}_{n=1,2,\ldots},$ with $b^{(n)}\in \Phii_\Bb (x^{(n)},a^{(n)})$ for each $n=1,2,\ldots,$ such that $b$ is a limit point of the sequence $\{b^{(n)}\}_{n=1,2,\ldots}.$
\end{itemize}
\end{definition}

The properties of $\Aa$-lower semi-continuous functions are described in Appendix.  In particular, this assumption is stronger than lower semi-continuity. According to Lemma~\ref{lem:unifAlsc}, this assumption holds a  for lower semi-continuous multifunction  $\Phii_{\Bb}:\Gr(\Phii_{\Aa})\subset\Xx\times\Aa\mapsto S(\Bb)$ in the following two cases: (i) the multifunction $\Phii_{\Aa}:\Xx\mapsto S(\Aa)$ is upper semi-continuous and compact-valued at each $x\in\Xx,$ and (ii) the sets $\Phii_{\Bb}(x,a)$ do not depend on $a\in 
 \Phii_{\Aa}(x)$ for all $x\in\Xx,$ as this takes place  for games with players making simultaneous decisions.

The following theorem establishes sufficient conditions for $\K$-inf-compactness of the worst-loss function $\ff^\sharp:\Gr(\Phii_{\Aa})\subset
\Xx\times\Aa\mapsto \overline{\mathbb{R}}$ defined in (\ref{eq1starworstloss}), when the image sets $\{\Phii_{\Aa}(x)\}_{x\in\Xx}$ and $\{\Phii_{\Bb}(x,a)\}_{(x,a)\in \Gr(\Phii_{\Aa})}$ can be noncompact. 
We remark that we currently do not know whether the assumption, that the set-valued mapping $\ff^{\Aa\leftrightarrow \Bb}:\Gr(\Phii_\Bb^{\Aa\leftrightarrow \Bb})\subset (\Xx\times\Bb)\times\Aa\mapsto\R$ is $\Aa$-lower semi-continuous, can be relaxed in Theorems~\ref{th:wloss_kinfcomp}, \ref{th:minimax_lsc}, \ref{th:minimax_cont}, 
\ref{th: uscphii1}, and \ref{th: BergeMinimax} to the assumption that this set-valued mapping is lower semi-continuous.
\begin{theorem}{\rm($\K$-inf-compactness of the worst-loss function)}\label{th:wloss_kinfcomp}
Let $\Phii_{\Bb}:\Gr(\Phii_{\Aa})\subset\Xx\times\Aa\mapsto S(\Bb)$ be an $\Aa$-lower semi-continuous set-valued mapping and the function $\ff^{\Aa\leftrightarrow \Bb}:\Gr(\Phii_\Bb^{\Aa\leftrightarrow \Bb})\subset (\Xx\times\Bb)\times\Aa\mapsto\R$ defined in (\ref{eq:auxil3}), where $\Phii_\Bb^{\Aa\leftrightarrow \Bb}$ is defined in (\ref{eq:auxil23}), be $\K$-inf-compact on $\Gr(\Phii_\Bb^{\Aa\leftrightarrow \Bb}).$ Then the worst-loss function $\ff^\sharp:
\Gr(\Phii_{\Aa})\subset\Xx\times\Aa\mapsto \overline{\mathbb{R}}$ defined in (\ref{eq1starworstloss}) is $\K$-inf-compact on $\Gr(\Phii_{\Aa}).$
\end{theorem}
\begin{proof}
Since the function $\ff^{\Aa\leftrightarrow \Bb}:\Gr(\Phii_\Bb^{\Aa\leftrightarrow \Bb})\subset (\Xx\times\Bb)\times\Aa\mapsto\R$ defined in (\ref{eq:auxil3}), where $\Phii_\Bb^{\Aa\leftrightarrow \Bb}$ is defined in (\ref{eq:auxil23}), is $\K$-inf-compact on $\Gr(\Phii_\Bb^{\Aa\leftrightarrow \Bb}),$ we have that properties (i) and (ii) from Remark~\ref{rem:substit} hold.

To prove that the function $\ff^\sharp$ is $\K$-inf-compact on $\Gr(\Phii_{\Aa}),$  we fix arbitrary $C\in\K(\X),$ $\lambda\in\mathbb{R},$ and $\{(x^{(n)}, a^{(n)})\}_{n=1,2,\ldots}\subset \Gr_C(\Phii_{\Aa})$ such that
\begin{equation}\label{eq:fu1}
\ff^\sharp(x^{(n)},a^{(n)})\le \lambda,
\end{equation}
for each $n=1,2,\ldots,$ and establish that the sequence $\{(x^{(n)}, a^{(n)})\}_{n=1,2,\ldots}$ has a limit point $(x,a)\in \Gr_C(\Phii_{\Aa})$ satisfying
$\ff^\sharp(x,a)\le \lambda.$

According to Theorem~\ref{cor:wloss_lsc}, it is sufficient to prove that
the sequence \\ $\{(x^{(n)}, a^{(n)})\}_{n=1,2,\ldots}\subset \Gr_C(\Phii_{\Aa})$ satisfying inequality (\ref{eq:fu1}) has a limit point $(x,a)\in \Gr_C(\Phii_{\Aa}).$
%
Indeed, since $C\in\K(\Xx),$  without loss of generality we may assume that
the sequence $\{x^{(n)}\}_{n=1,2,\ldots}$ converges in $\Xx$ and its limit $x$ belongs to $C.$
To prove that the sequence $\{a^{(n)}
\}_{n=1,2,\ldots}$ has a limit point $a\in \Phii_{\Aa}(x),$ we fix an arbitrary $b\in \Phii_{\Bb}(x,a^*)$ for some $a^*\in \Phii_{\Aa}(x)$ and note that there exists a sequence $\{b^{(n)}\}_{n=1,2,\ldots}$ with $b^{(n)}\in \Phii_{\Bb}(x^{(n)},a^{(n)}),$ $n=1,2,\ldots,$ that converges and its limit equals to $b$ because the set-valued mapping $\Phii_{\Bb}:\Gr(\Phii_{\Aa})\subset\Xx\times\Aa\mapsto S(\Bb)$ is   $\Aa$-lower semi-continuous. Then, according to (\ref{eq1starworstloss}) and (\ref{eq:fu1}),  the sequence $\{\ff(x^{(n)},a^{(n)},b^{(n)}) \}_{n=1,2,\ldots}$ is
bounded above by $\lambda.$ Therefore, property (ii) from Remark~\ref{rem:substit} implies that the sequence $\{a^{(n)}
\}_{n=1,2,\ldots}$ has a limit point $a\in \Phii_{\Aa}(x).$ Therefore, the sequence $\{(x^{(n)}, a^{(n)})\}_{n=1,2,\ldots}$ has a limit point $(x,a)\in \Gr_C(\Phii_{\Aa}).$
%
\qed \end{proof}


 The following theorem establishes sufficient conditions for upper semi-continuity of the worst-loss function $\ff^\sharp:
\Gr(\Phii_{\Aa})\subset
\Xx\times\Aa\mapsto \overline{\mathbb{R}}$ defined in (\ref{eq1starworstloss}) and basic properties for the solution multifunction $\Phii_{\Bb}^*$ defined in (\ref{e:defFi*minimax2}), when the image sets $\{\Phii_{\Aa}(x)\}_{x\in\Xx}$ and $\{\Phii_{\Bb}(x,a)\}_{(x,a)\in \Gr(\Phii_{\Aa})}$ can be noncompact.

\begin{theorem}{\rm(Upper semi-continuity of the worst-loss function)}\label{th:wloss_usc}
If a function $\ff:{\rm Gr}(\Phii_{\Bb})\subset (\Xx \times \Aa)\times\Bb \mapsto \overline{\mathbb{R}}$ is $\K$-sup-compact on $\Gr(\Phii_{\Bb}),$ then the worst-loss function $\ff^\sharp:
\Gr(\Phii_{\Aa})\subset\Xx\times\Aa\mapsto \overline{\mathbb{R}}$ defined in (\ref{eq1starworstloss}) is upper semi-continuous. Moreover, the supremum in (\ref{eq1starworstloss}) can be replaced with the maximum and the nonempty sets $\{\Phii_{\Bb}^*(x,a)\}_{(x,a)\in\Gr(\Phii_{\Aa})}$ defined in (\ref{e:defFi*minimax2}) (see also the last equality in (\ref{eq:equal}))
satisfy the following properties:
\begin{itemize}
\item[{\rm(a)}] the graph ${\rm Gr}(\Phii_{\Bb}^*)$ is a Borel subset of  $\Xx\times \Aa\times\Bb;$
\item[{\rm(b)}] if $\ff^\sharp(x,a)=-\infty,$
then $\Phii_{\Bb}^*(x,a)=\Phii_{\Bb}(x,a),$ and, if $\ff^\sharp(x,a)>-\infty,$ then
$\Phii_{\Bb}^*(x,a)$ is compact.
\end{itemize}
\end{theorem}
\begin{proof}
Since the function $\ff:{\rm Gr}(\Phii_{\Bb})\subset (\Xx \times \Aa)\times\Bb \mapsto \overline{\mathbb{R}}$ is $\K$-sup-compact on $\Gr(\Phii_{\Bb}),$ we have that Theorem~\ref{th:BGL}, being applied to $\X=\Xx\times\Aa,$ $\Y=\Bb,$ $\Phi=\Phii_{\Bb},$ and $f=-\ff,$ implies all the statements of Theorem~\ref{th:wloss_usc}.
\qed \end{proof}

The following theorem describes sufficient conditions for continuity of the worst-loss function $\ff^\sharp:
\Gr(\Phii_{\Aa})\subset
\Xx\times\Aa\mapsto \overline{\mathbb{R}}$ defined in (\ref{eq1starworstloss}), when the image sets $\{\Phii_{\Aa}(x)\}_{x\in\Xx}$ and $\{\Phii_{\Bb}(x,a)\}_{(x,a)\in \Gr(\Phii_{\Aa})}$ can be noncompact.
\begin{theorem}{\rm(Continuity of the worst-loss function)}\label{th:wloss_cont}
Let $\Phii_{\Bb}:\Gr(\Phii_{\Aa})\subset\Xx\times\Aa\mapsto S(\Bb)$ be a lower semi-continuous set-valued mapping,  $\ff:{\rm Gr}(\Phii_{\Bb})\subset (\Xx \times \Aa)\times\Bb \mapsto \overline{\mathbb{R}}$ be a $\K$-sup-compact function on $\Gr(\Phii_{\Bb}),$ and the function $\ff^{\Aa\leftrightarrow \Bb}:\Gr(\Phii_\Bb^{\Aa\leftrightarrow \Bb})\subset (\Xx\times\Bb)\times\Aa\mapsto\R$ defined in (\ref{eq:auxil3}), where $\Phii_\Bb^{\Aa\leftrightarrow \Bb}$ is defined in (\ref{eq:auxil23}), be $\K$-inf-compact on $\Gr(\Phii_\Bb^{\Aa\leftrightarrow \Bb}).$ Then the worst-loss function $\ff^\sharp:
\Gr(\Phii_{\Aa})\subset\Xx\times\Aa\mapsto \overline{\mathbb{R}}$ defined in (\ref{eq1starworstloss}) is continuous.
\end{theorem}
\begin{proof}
Theorem~\ref{cor:wloss_lsc} implies that the worst-loss function $\ff^\sharp:
\Gr(\Phii_{\Aa})\subset\Xx\times\Aa\mapsto \overline{\mathbb{R}}$ is
 lower semi-continuous. Theorem~\ref{th:wloss_usc} implies that $\ff^\sharp:
\Gr(\Phii_{\Aa})\subset\Xx\times\Aa\mapsto \overline{\mathbb{R}}$ is upper semi-continuous. Therefore, $\ff^\sharp:
\Gr(\Phii_{\Aa})\subset\Xx\times\Aa\mapsto \overline{\mathbb{R}}$  is continuous.
\qed \end{proof}

The following theorem describes sufficient conditions for lower semi-continuity of the minimax function $\vv^\sharp$ defined in (\ref{eq1starminimax}) and basic properties for the solution multifunction $\Phii_{\Aa}^*$ defined in (\ref{e:defFi*minimax1}), when the image sets $\{\Phii_{\Aa}(x)\}_{x\in\Xx}$ and $\{\Phii_{\Bb}(x,a)\}_{(x,a)\in \Gr(\Phii_{\Aa})}$ can be noncompact.

\begin{theorem}{\rm(Lower semi-continuity of minimax)}\label{th:minimax_lsc}
Let $\Phii_{\Bb}:\Gr(\Phii_{\Aa})\subset\Xx\times\Aa\mapsto S(\Bb)$ be an   $\Aa$-lower semi-continuous set-valued mapping and the function $\ff^{\Aa\leftrightarrow \Bb}:\Gr(\Phii_\Bb^{\Aa\leftrightarrow \Bb})\subset (\Xx\times\Bb)\times\Aa\mapsto\R$ defined in (\ref{eq:auxil3}), where $\Phii_\Bb^{\Aa\leftrightarrow \Bb}$ is defined in (\ref{eq:auxil23}), be $\K$-inf-compact on $\Gr(\Phii_\Bb^{\Aa\leftrightarrow \Bb}).$ Then the minimax function $\vv^\sharp:
\Xx\mapsto \overline{\mathbb{R}}$ defined in (\ref{eq1starminimax}) is lower semi-continuous.
Moreover, the infimum in (\ref{eq1starminimax}) can be
replaced with the minimum and the nonempty sets $\{\Phii_{\Aa}^*(x)\}_{x\in\Xx}$ defined in (\ref{e:defFi*minimax1})
satisfy the following properties:
\begin{itemize}
\item[{\rm(a)}] the graph ${\rm Gr}(\Phii_{\Aa}^*)$ is a Borel subset of  $\Xx\times \Aa;$
\item[{\rm(b)}] if $\vv^\sharp(x)=+\infty,$
then $\Phii_{\Aa}^*(x)=\Phii_{\Aa}(x),$ and, if $\vv^\sharp(x)<+\infty,$ then
$\Phii_{\Aa}^*(x)$ is compact.
\end{itemize}
\end{theorem}
\begin{proof}
Theorem~\ref{th:wloss_kinfcomp} implies that the worst-loss function $\ff^\sharp:
\Gr(\Phii_{\Aa})\subset\Xx\times\Aa\mapsto \overline{\mathbb{R}}$ defined in (\ref{eq1starworstloss}) is $\K$-inf-compact on $\Gr(\Phii_{\Aa}).$ Therefore, Theorem~\ref{th:BGL}, being applied to $\X:=\Xx,$ $\Y:=\Aa,$ $\Phi:=\Phii_{\Aa},$ and $f:=\ff^\sharp,$ implies all the statements of Theorem~\ref{th:minimax_lsc}.
\qed \end{proof}

The following theorem describes sufficient conditions for upper semi-continuity of the minimax function $\vv^\sharp$ defined in (\ref{eq1starminimax}) and basic properties for the solution multifunction $\Phii_{\Bb}^*$ defined in (\ref{e:defFi*minimax2}), when the image sets $\{\Phii_{\Aa}(x)\}_{x\in\Xx}$ and $\{\Phii_{\Bb}(x,a)\}_{(x,a)\in \Gr(\Phii_{\Aa})}$ can be noncompact.
\begin{theorem}{\rm(Upper semi-continuity of minimax)}\label{th:minimax_usc}
Let $\Phii_{\Aa}:\Xx\mapsto S(\Aa)$ be a lower semi-continuous set-valued mapping and $\ff:{\rm Gr}(\Phii_{\Bb})\subset (\Xx \times \Aa)\times\Bb \mapsto \overline{\mathbb{R}}$ be a $\K$-sup-compact function on $\Gr(\Phii_{\Bb}).$ Then the minimax function $\vv^\sharp:
\Xx\mapsto \overline{\mathbb{R}}$ defined in (\ref{eq1starminimax}) is upper semi-continuous. Moreover, the supremums in (\ref{eq1starworstloss}) and (\ref{eq1starminimax}) can be replaced with the maximums and the nonempty sets $\{\Phii_{\Bb}^*(x,a)\}_{(x,a)\in\Gr(\Phii_{\Aa})}$ defined in (\ref{e:defFi*minimax2}) (see also the last equality in (\ref{eq:equal}))
satisfy properties (a) and (b) of Theorem~\ref{th:wloss_usc}.
\end{theorem}
\begin{proof}
Theorem~\ref{th:wloss_usc} implies that the worst-loss function $\ff^\sharp:
\Gr(\Phii_{\Aa})\subset\Xx\times\Aa\mapsto \overline{\mathbb{R}}$ defined in (\ref{eq1starworstloss}) is upper semi-continuous on $\Gr(\Phii_{\Aa}),$  the supremums in (\ref{eq1starworstloss}) and (\ref{eq1starminimax}) can be replaced with the maximums, and the nonempty sets $\{\Phii_{\Bb}^*(x,a)\}_{(x,a)\in\Gr(\Phii_{\Aa})}$ defined in (\ref{e:defFi*minimax2}) (see also the last equality in (\ref{eq:equal}))
satisfy properties (a) and (b) of Theorem~\ref{th:wloss_usc}. The upper semi-continuity of the  minimax function $\vv^\sharp:
\Xx\mapsto \overline{\mathbb{R}}$ follows from Theorem~\ref{th:BGU}, being applied to $\X:=\Xx,$ $\Y:=\Aa,$ $\Phi:=\Phii_{\Aa},$ and $f:=\ff^\sharp,$ because a set-valued mapping $\Phii_{\Aa}:\Xx\mapsto S(\Aa)$ is lower semi-continuous and the function $\ff^\sharp:
\Gr(\Phii_{\Aa})\subset\Xx\times\Aa\mapsto \overline{\mathbb{R}}$ is upper semi-continuous.
\qed \end{proof}

The following theorem describes sufficient conditions for continuity of the minimax function $\vv^\sharp$ defined in (\ref{eq1starminimax}), when the image sets $\{\Phii_{\Aa}(x)\}_{x\in\Xx}$ and $\{\Phii_{\Bb}(x,a)\}_{(x,a)\in \Gr(\Phii_{\Aa})}$ can be noncompact.

\begin{theorem}{\rm(Continuity of minimax)}\label{th:minimax_cont}
Let $\Phii_{\Aa}:\Xx\mapsto S(\Aa)$ be a lower semi-continuous set-valued mapping, $\Phii_{\Bb}:\Gr(\Phii_{\Aa})\subset\Xx\times\Aa\mapsto S(\Bb)$ be an   $\Aa$-lower semi-continuous set-valued mapping, $\ff:{\rm Gr}(\Phii_{\Bb})\subset (\Xx \times \Aa)\times\Bb \mapsto \overline{\mathbb{R}}$ be a $\K$-sup-compact function on $\Gr(\Phii_{\Bb}),$ and the function $\ff^{\Aa\leftrightarrow \Bb}:\Gr(\Phii_\Bb^{\Aa\leftrightarrow \Bb})\subset (\Xx\times\Bb)\times\Aa\mapsto\R$ defined in (\ref{eq:auxil3}), where $\Phii_\Bb^{\Aa\leftrightarrow \Bb}$ is defined in (\ref{eq:auxil23}), be $\K$-inf-compact on $\Gr(\Phii_\Bb^{\Aa\leftrightarrow \Bb}).$ Then the minimax function $\vv^\sharp:
\Xx\mapsto \overline{\mathbb{R}}$ defined in (\ref{eq1starminimax}) is continuous.
\end{theorem}
\begin{proof}
Theorem~\ref{th:minimax_lsc} implies that the function $\vv^\sharp:
\Xx\mapsto \overline{\mathbb{R}}$ is lower semi-continuous. Theorem~\ref{th:minimax_usc} implies that the function $\vv^\sharp:
\Xx\mapsto \overline{\mathbb{R}}$ is upper semi-continuous. Thus, the function $\vv^\sharp:
\Xx\mapsto \overline{\mathbb{R}}$ is continuous.
\qed \end{proof}

The following theorem describes sufficient conditions for  
$\K$-upper semi-compactness of the solution multifunction $\Phii_{\Aa}^*$  defined in (\ref{e:defFi*minimax1}), when the image sets $\{\Phii_{\Aa}(x)\}_{x\in\Xx}$ and $\{\Phii_{\Bb}(x,a)\}_{(x,a)\in \Gr(\Phii_{\Aa})}$ can be noncompact.
\begin{theorem}{\rm(Continuity properties for solution multifunction $\Phii_{\Aa}^*$)}\label{th: uscphii1}
Let $\vv^\sharp:
\Xx\mapsto \mathbb{R}\cup\{-\infty\}$ defined in (\ref{eq1starminimax}) be a continuous function, $\Phii_{\Bb}:\Gr(\Phii_{\Aa})\subset\Xx\times\Aa\mapsto S(\Bb)$ be an   $\Aa$-lower semi-continuous set-valued mapping, and the function $\ff^{\Aa\leftrightarrow \Bb}:\Gr(\Phii_\Bb^{\Aa\leftrightarrow \Bb})\subset (\Xx\times\Bb)\times\Aa\mapsto\R$ defined in (\ref{eq:auxil3}), where $\Phii_\Bb^{\Aa\leftrightarrow \Bb}$ is defined in (\ref{eq:auxil23}), be $\K$-inf-compact on $\Gr(\Phii_\Bb^{\Aa\leftrightarrow \Bb}).$ Then the infimum in (\ref{eq1starminimax}) can be
replaced with the minimum and the solution multifunction $\Phii_{\Aa}^*:\Xx\mapsto S(\Aa)$ defined in (\ref{e:defFi*minimax1}) is upper semi-continuous and compact-valued.
\end{theorem}
\begin{proof}
Theorem~\ref{th:wloss_kinfcomp} implies that the worst-loss function $\ff^\sharp:
\Gr(\Phii_{\Aa})\subset\Xx\times\Aa\mapsto \overline{\mathbb{R}}$ defined in (\ref{eq1starworstloss}) is $\K$-inf-compact on $\Gr(\Phii_{\Aa}).$ Since $\vv^\sharp:
\Xx\mapsto \mathbb{R}\cup\{-\infty\}$ defined in (\ref{eq1starminimax}) is a continuous function, we have that  Theorem~\ref{th:usc}, being applied to $\X:=\Xx,$ $\Y:=\Aa,$ $\Phi:=\Phii_{\Aa},$ and $f:=\ff^\sharp,$ implies that the infimum in (\ref{eq1starminimax}) can be
replaced with the minimum and the solution multifunction $\Phii_{\Aa}^*:\Xx\mapsto S(\Aa)$ defined in (\ref{e:defFi*minimax1}) is upper semi-continuous and compact-valued.
\qed \end{proof}

The following theorem provides sufficient conditions for $\K$-upper semi-compactness of the solution multifunction $\Phii_{\Bb}^*$ defined in (\ref{e:defFi*minimax2}), when the image sets $\{\Phii_{\Aa}(x)\}_{x\in\Xx}$ and $\{\Phii_{\Bb}(x,a)\}_{(x,a)\in \Gr(\Phii_{\Aa})}$ can be noncompact.
\begin{theorem}{\rm(Continuity properties of the solution multifunction $\Phii_{\Bb}^*$)}\label{th: uscphii2}
Let $\ff^\sharp:
\Gr(\Phii_{\Aa})\subset\Xx\times\Aa\mapsto {\mathbb{R}}\cup\{+\infty\}$ defined in (\ref{eq1starworstloss}) be a continuous function on $\Gr(\Phii_{\Aa})$ and  $\ff:{\rm Gr}(\Phii_{\Bb})\subset (\Xx \times \Aa)\times\Bb \mapsto \overline{\mathbb{R}}$ be a $\K$-sup-compact function on $\Gr(\Phii_{\Bb}).$ Then the supremums in (\ref{eq1starworstloss}) and (\ref{eq1starminimax}) can be replaced with the maximums and the solution multifunction $\Phii_{\Bb}^*:\Gr(\Phii_{\Aa})\subset\Xx\times\Aa\mapsto S(\Bb)$ defined in (\ref{e:defFi*minimax2}) is upper semi-continuous and compact-valued.
\end{theorem}
\begin{proof}
According to Remark~\ref{rem:comp}, the statements of the theortem follow from Theorem~\ref{th:usc}, being applied to $\X:=\Gr(\Phii_{\Aa}),$ $\Y:=\Bb,$ $\Phi:=\Phii_{\Bb},$ and $f:=-\ff.$
\qed \end{proof}

For metric spaces the following theorem can be viewed as an extension of Berge's maximum theorem for noncompact image sets from Feinberg et al. \cite[Theorem~1.4]{FKV} to the minimax formulation.
\begin{theorem}{\rm(Continuity of the worst-loss function $\ff^\sharp$ and the minimax function $\vv^\sharp$ and upper semi-continuity of the solution multifunctions $\Phii_{\Aa}^*$ and $\Phii_{\Bb}^*$)}\label{th: BergeMinimax}
Let $\Phii_{\Aa}:\Xx\mapsto S(\Aa)$ be a lower semi-continuous set-valued mapping,  $\Phii_{\Bb}:\Gr(\Phii_{\Aa})\subset\Xx\times\Aa\mapsto S(\Bb)$ be an   $\Aa$-lower semi-continuous set-valued mapping, $\ff:{\rm Gr}(\Phii_{\Bb})\subset (\Xx \times \Aa)\times\Bb \mapsto {\mathbb{R}}$ be a $\K$-sup-compact function on $\Gr(\Phii_{\Bb}),$ and the function $\ff^{\Aa\leftrightarrow \Bb}:\Gr(\Phii_\Bb^{\Aa\leftrightarrow \Bb})\subset (\Xx\times\Bb)\times\Aa\mapsto\mathbb{R}$ defined in (\ref{eq:auxil3}), where $\Phii_\Bb^{\Aa\leftrightarrow \Bb}$ is defined in (\ref{eq:auxil23}), be $\K$-inf-compact on $\Gr(\Phii_\Bb^{\Aa\leftrightarrow \Bb}).$ Then the worst-loss function $\ff^\sharp:
\Gr(\Phii_{\Aa})\subset\Xx\times\Aa\mapsto {\mathbb{R}}$ defined in (\ref{eq1starworstloss}) is continuous and the minimax function $\vv^\sharp:
\Xx\mapsto {\mathbb{R}}$ defined in (\ref{eq1starminimax}) is continuous. Moreover, the following two properties hold:
\begin{itemize}
\item[{\rm(a)}] the infimum in (\ref{eq1starminimax}) can be
replaced with the minimum, and the solution multifunction $\Phii_{\Aa}^*:\Xx\mapsto S(\Aa)$ defined in (\ref{e:defFi*minimax1}) is upper semi-continuous and compact-valued;
\item[{\rm(b)}] the supremums in (\ref{eq1starworstloss}) and (\ref{eq1starminimax}) can be replaced with the maximums, and the solution multifunction $\Phii_{\Bb}^*:\Gr(\Phii_{\Aa})\subset\Xx\times\Aa\mapsto S(\Bb)$ defined in (\ref{e:defFi*minimax2}) is upper semi-continuous and compact-valued.
\end{itemize}
\end{theorem}
\begin{proof}
Theorem~\ref{th:wloss_cont} implies that the worst-loss function $\ff^\sharp:
\Gr(\Phii_{\Aa})\subset\Xx\times\Aa\mapsto {\mathbb{R}}$ defined in (\ref{eq1starworstloss}) is continuous on $\Gr(\Phii_{\Aa}).$
Continuity of the minimax function $\vv^\sharp:
\Xx\mapsto {\mathbb{R}}$ defined in (\ref{eq1starminimax})
follows from Theorem~\ref{th:minimax_cont}. Theorems~\ref{th: uscphii1} and \ref{th: uscphii2} imply statements (a) and (b) respectively.
\qed \end{proof}

\section{Preserving Properties of $\K$-inf-compact functions}\label{s4}

In Section~\ref{sec:minimax} we considered problems in which players select actions deterministically. In other words, players play pure strategies.  The previous section describes the continuity properties for objective functions and solution multifunctions for such problems with possibly unbounded payoffs, and noncompact action sets. In general, it is known that, if the second player knows the decision of the first players, pure strategies are sufficient.  In Section~\ref{sec:perf} we show that pure strategies are indeed  sufficient for the problem studied in the previous section.  However, if players make decisions simultaneously, pure strategies usually are not sufficient, and the players should choose randomized strategies, which are probability distributions on the sets of actions. The remarkable fact is that the property of $\K$-inf-compactness is preserved when randomized strategies are used instead of pure ones. This section describes such results.  Most of them were derived in Feinberg et al. \cite{POMDP} for studying partially observable Markov decision processes.

Let $\mathbb{S}$ be a metric space. An integral $\int_\mathbb{S} f (s)\mu(ds)$ of a measurable ${\R}$-valued function $f$ on $\mathbb{S}$ over the measure $\mu \in
\mathbb{P}(\mathbb{S})$ is well-defined if either $\int_\mathbb{S} f^- (s)\mu(ds)>-\infty$ or $\int_\mathbb{S} f^+ (s)\mu(ds)<+\infty, $  where for $s\in \mathbb{S}$ 
\[f^-(s)=\min\{f(s),0\},\quad f^+(s)=\max\{f(s),0\}.\]   If the integral is well-defined, then \[\int_\mathbb{S} f (s)\mu(ds):=\int_\mathbb{S} f^+ (s)\mu(ds)+ \int_\mathbb{S} f^-(s)\mu(ds).\]

Let ${\mathcal B}(\mathbb{S})$ be a Borel
$\sigma$-field on $\mathbb{S},$ that is, the $\sigma$-field generated by all
open sets of the metric space $\mathbb{S}.$  For a nonempty Borel subset $S\subset \mathbb{S},$ denote by
${\mathcal B}(S)$ the $\sigma$-field whose elements are
intersections of $S$ with elements of ${\mathcal B}(\mathbb{S}).$  Observe
that $S$ is a metric space with the same metric as on $\mathbb{S},$ and
${\mathcal B}(S)$ is its Borel $\sigma$-field. For a metric space
$\mathbb{S},$ let $\mathbb{P}(\mathbb{S})$ be the set of probability measures on
$(\mathbb{S},{\mathcal B}(\mathbb{S}))$ and $\mathbb{P}^{fs}(\mathbb{S})$ denote the set of all probability measures whose supports are finite subsets of $\mathbb{S}.$ A sequence of probability measures
$\{\mu^{(n)}\}_{n=1,2,\ldots}$ from $\mathbb{P}(\mathbb{S})$ \textit{converges weakly} to $\mu\in\mathbb{P}(\mathbb{S})$ if for
each bounded continuous function $f$ on $\mathbb{S}$
\[\int_\mathbb{S} f(s)\mu^{(n)}(ds)\to \int_\mathbb{S} f(s)\mu(ds) \qquad {\rm as \quad
}n\to\infty.
\]
Note that the set $\mathbb{P}^{fs}(\mathbb{S})$ is dense in a separable metric space $\mathbb{P}(\mathbb{S})$ with respect to the weak convergence topology for probability measures, when $\mathbb{S}$ is
a separable metric space; Parthasarathy \cite[Chapter II, Theorem 6.3]{Part}.

Let $\X, \Y$ be nonempty Borel subsets of respective Polish spaces (complete separable metric spaces). The following lemma, three theorems, and a corollary describe preserving properties for lower semi-continuous, inf-compact,  and $\K$-inf-compact functions.

\begin{lemma}{\rm(Feinberg et al. \cite[Lemma~6.1]{POMDP})}\label{lem:lsc}
If a function $f:\X\times\Y\mapsto \mathbb{R}\cup\{+\infty\}$ is bounded from below and lower 
semi-continuous, then the function ${\hat{f}}:\X\times\mathbb{P}(\Y)\mapsto\mathbb{R}\cup\{+\infty\},$
\begin{equation}\label{eq:presprop}
{\hat{f}}(x,z):=\int_{\Y}f(x,y)z(dy),\quad x\in\X,\,z\in\mathbb{P}(\Y),
\end{equation}
is bounded from below by the same constant as $f$ 
and lower semi-continuous.
\end{lemma}

\begin{theorem}{\rm(Feinberg et al. \cite[Theorem~6.1]{POMDP})}\label{th:inf-comp}
If $f:\X\times\Y\mapsto \mathbb{R}\cup\{+\infty\}$ is an inf-compact function on
$\X\times\Y,$ then the function ${\hat{f}}:\X\times\mathbb{P}(\Y)\mapsto\mathbb{R}\cup\{+\infty\}$ defined
in (\ref{eq:presprop}) is inf-compact  on $\X\times\mathbb{P}(\Y).$
\end{theorem}

\begin{corollary}\label{cor:K-inf-comp}
If $f:\X\times\Y\mapsto \mathbb{R}\cup\{+\infty\}$ is a $\K$-inf-compact function on
$\X\times\Y,$ then the function ${\hat{f}}:\X\times\mathbb{P}(\Y)\mapsto\mathbb{R}\cup\{+\infty\}$ defined
in (\ref{eq:presprop}) is $\K$-inf-compact on $\X\times\mathbb{P}(\Y).$
\end{corollary}
\begin{proof}
According to Definition~\ref{def:kinf}, the function ${\hat{f}}:\X\times\mathbb{P}(\Y)\mapsto\mathbb{R}\cup\{+\infty\}$ defined
in (\ref{eq:presprop}) is $\K$-inf-compact on $\X\times\mathbb{P}(\Y)$ if and only if for every $C\in \K(\X)$ this function is inf-compact on $C\times \mathbb{P}(\Y).$

Let us prove that the function ${\hat{f}}:\X\times\mathbb{P}(\Y)\mapsto\mathbb{R}\cup\{+\infty\}$ defined
in (\ref{eq:presprop}) is inf-compact on $C\times \mathbb{P}(\Y)$ for each $C\in \K(\X).$ For this purpose we fix an arbitrary $C\in \K(\X)$ and note that the function $f\big|_C:C\times\Y\mapsto \mathbb{R}\cup\{+\infty\}$ is inf-compact on $C\times\Y$ because this function  is $\K$-inf-compact on
$\X\times\Y.$  Theorem~\ref{th:inf-comp} implies that the function ${\hat{f}}$ defined
in (\ref{eq:presprop}) is inf-compact on $C\times\mathbb{P}(\Y).$ Therefore, this function is $\K$-inf-compact on $\X\times\mathbb{P}(\Y)$ since $C\in \K(\X)$ is arbitrary. 
\qed \end{proof}

\begin{theorem}{\rm(Feinberg et al. \cite[Theorem~3.3]{POMDP})}\label{th:K-inf-comp}
If the function $f:\X\times\Y\mapsto \mathbb{R}\cup\{+\infty\}$ is bounded from below and $\K$-inf-compact on
$\X\times\Y,$ then the function $\bar{f}:\mathbb{P}(\X)\times\Y\mapsto\mathbb{R}\cup\{+\infty\},$
\[
\bar{f}(z,y):=\int_{\X}f(x,y)z(dx),\quad z\in\mathbb{P}(\X),\,y\in\Y,
\]
is bounded from below by the same constant as $f$ and $\K$-inf-compact on $\mathbb{P}(\X)\times \Y. $
\end{theorem}

\begin{theorem}\label{th:K-inf-comp_two_var}
If the function $f:\X\times\Y\mapsto \mathbb{R}\cup\{+\infty\}$ is bounded from below and $\K$-inf-compact on
$\X\times\Y,$ then the function $\tilde{f}:\mathbb{P}(\X)\times\mathbb{P}(\Y)\mapsto\mathbb{R}\cup\{+\infty\},$
\begin{equation}\label{eq:bar}
\tilde{f}(z^\X,z^\Y):=\int_{\X}\int_\Y f(x,y)z^\Y(dy)z^\X(dx),\quad z^\X\in\mathbb{P}(\X),\,z^\Y\in\mathbb{P}(\Y),
\end{equation}
is bounded from below by the same constant as $f$ and $\K$-inf-compact on $\mathbb{P}(\X)\times \mathbb{P}(\Y).$ 
\end{theorem}
\begin{proof}
Lemma~\ref{lem:lsc}, being applied to $f:\X\times\Y\mapsto \mathbb{R}\cup\{+\infty\},$ implies that
the function ${\hat{f}}:\X\times\mathbb{P}(\Y)\mapsto\mathbb{R}\cup\{+\infty\}$ defined in (\ref{eq:presprop})
is bounded from below by the same constant as $f.$ Then, Lemma~\ref{lem:lsc}, being applied to
${\hat{f}}:\X\times\mathbb{P}(\Y)\mapsto\mathbb{R}\cup\{+\infty\},$ implies that
the function $\tilde{f}:\mathbb{P}(\X)\times\mathbb{P}(\Y)\mapsto\mathbb{R}\cup\{+\infty\}$ is bounded from below by the same constant as $f.$

Theorem~\ref{th:K-inf-comp}, being applied to $f:\X\times\Y\mapsto \mathbb{R}\cup\{+\infty\},$ implies that the function $\bar{f}:\mathbb{P}(\X)\times\Y\mapsto\mathbb{R}\cup\{+\infty\}$ defined in (\ref{eq:bar}) is $\K$-inf-compact on $\mathbb{P}(\X)\times \Y.$ Therefore, Corollary~\ref{cor:K-inf-comp}, being applied to $\bar{f}:\mathbb{P}(\X)\times\Y\mapsto\mathbb{R}\cup\{+\infty\},$ implies that the function $\tilde{f}:\mathbb{P}(\X)\times\mathbb{P}(\Y)\mapsto\mathbb{R}\cup\{+\infty\}$
is $\K$-inf-compact on $\mathbb{P}(\X)\times \mathbb{P}(\Y).$
\qed \end{proof}

\section{Two-Person Zero-Sum Games with Simultaneous Moves}\label{s5}


In this section we  provide sufficient conditions for continuity of the lopsided value functions, upper semi-continuity of solution multifunctions, and compactness of solution sets for zero-sum stochastic games with possibly uncountable and noncompact action sets and unbounded payoff functions.

\subsection{Preliminaries}\label{sub:prem}


\begin{definition}\label{defi:game}  A \textit{  two-person zero-sum game} is a triplet $\{\A,\bb, c\},$ where
  \begin{itemize}
\item[{\rm(i)}] $\A$ is the \textit{space of actions for Player I}, which is a nonempty Borel subset of a Polish space;
\item[{\rm(ii)}] $\bb$ is the \textit{space of actions for Player II}, which is a nonempty Borel subset of a Polish space;
\item[{\rm(iii)}] the \textit{payoff} to  Player II, $-\infty< c(a,b)< +\infty,$ for choosing actions $a\in \A$ and $b\in \bb,$ is a \textit{measurable} function on $\A\times\bb;$
    \item[{\rm(iv)}] for each $b\in\bb$ the function $a\mapsto c(a,b)$ is bounded from below on $\A;$
\item[{\rm(v)}]  for each $a\in\A$ the function $b\mapsto c(a,b)$ is bounded from above on $\bb.$
\end{itemize}
\end{definition}
\begin{remark}\label{rem:sim}
{\rm If a triplet $\{\A,\bb, c\}$ is a two-person zero-sum game as defined above, then the triplet $\{\bb,\A,-c^{\A\leftrightarrow\bb}\},$ where $c^{\A\leftrightarrow\bb}(b,a)=c(a,b)$ for each $a\in\A$ and $b\in\bb,$ is also a two-person zero-sum game satisfying conditions in Definition~\ref{defi:game}.}
\end{remark}
\textit{The game is played as follows}:


$\bullet$  the decision-makers (Players I and II) choose simultaneously respective actions $a\in \A$ and $b\in\bb;$

$\bullet$ the result $(a,b)$ is announced to both of them;

$\bullet$ Player I pays Player II the amount $c(a,b).$

\textit{Strategies} (sometimes called mixed strategies) for Players I and II are probability measures
$\pi^\A\in \mathbb{P}( \mathbb{A})$ and $\pi^\bb\in \mathbb{P}( \mathbb{B}).$ 
 Moreover, a strategy  $\pi^\A$ ($\pi^\bb$)
is called \textit{pure}, if the probability measure
$\pi^\A(\,\cdot\,)$ ($\pi^\bb(\,\cdot\,)$) is concentrated at a point. 
Note that $\mathbb{P}(\A)$
is the
\textit{set of strategies} for Player I, 
and $\mathbb{P}(\bb)$ is the
\textit{set of strategies} for Player II. 

\begin{remark}\label{rem:ma}
{\rm
Assumptions (iv) and (v) for the game $\{\A,\bb,c\}$ are natural because without them the expected payoffs may be undefined even if one of the players chooses a pure strategy.
} 
\end{remark}

Let us set
\[
\begin{aligned}{\hat{c}}^\oplus(\pi^\A,\pi^\bb)&:=\int_{\A}\int_{\bb}c^+(a,b)\pi^\bb(db)\pi^\A(da),\\
&{\hat{c}}^\ominus(\pi^\A,\pi^\bb):=\int_{\A}\int_{\bb}c^-(a,b)\pi^\bb(db)\pi^\A(da),
\end{aligned}
\]
for each $(\pi^\A,\pi^\bb)\in \mathbb{P}(\A)\times\mathbb{P}(\bb).$ 
Then
the \textit{expected
payoff} to  Player II
\[
{\hat{c}}(\pi^\A,\pi^\bb):= {\hat{c}}^\oplus(\pi^\A,\pi^\bb)+{\hat{c}}^\ominus(\pi^\A,\pi^\bb),
\]
is well-defined if either ${\hat{c}}^\oplus(\pi^\A,\pi^\bb)<+\infty$ or ${\hat{c}}^\ominus(\pi^\A,\pi^\bb)>-\infty;$ $(\pi^\A,\pi^\bb)\in \mathbb{P}(\A)\times\mathbb{P}(\bb).$  Of course,
when the function $c$ is unbounded both below as well as above, the quantity ${\hat{c}}(\pi^\A,\pi^\bb)$ can be \textit{undefined} for some 
$(\pi^\A,\pi^\bb)\in\mathbb{P}(\A)\times\mathbb{P}(\bb).$ We denote  
\[
\begin{aligned}
&\mathbb{P}^{S}_{\pi^\A}(\bb):=\{\pi^\bb\in\mathbb{P}(\bb)\,:\, {\hat{c}}(\pi^\A,\pi^\bb)\mbox{ is well-defined}\},\quad \pi^\A\in\mathbb{P}(\A);\\
&\mathbb{P}^{S}_{\pi^\bb}(\A):=\{\pi^\A\in\mathbb{P}(\A)\,:\, {\hat{c}}(\pi^\A,\pi^\bb)\mbox{ is well-defined}\},\quad \pi^\bb\in\mathbb{P}(\bb).
\end{aligned}
\]

Further, if a measure $\pi^\A\in\mathbb{P}(\A)$ is concentrated at a point $a\in \A,$ then
we will write ${\hat{c}}(a,\pi^\bb)$ instead of ${\hat{c}}(\pi^\A,\pi^\bb)$ for each $\pi^\bb \in \mathbb{P}(\bb).$ Similarly, if a measure $\pi^\bb\in\mathbb{P}(\bb)$ is concentrated at a point $b\in \bb,$ then
we will write ${\hat{c}}(\pi^\A,b)$ instead of ${\hat{c}}(\pi^\A,\pi^\bb)$ for each $\pi^\A \in \mathbb{P}(\A).$

\begin{remark}\label{rem:new1}
{\rm Assumption (iv) for the game $\{\A,\bb,c\}$ implies that
 ${\hat{c}}^\ominus(\pi^\A,b)>-\infty$ for each $\pi^\A\in \mathbb{P}(\A)$ and $b\in\bb.$ Therefore,
$\mathbb{P}^{fs}(\bb)\subset \mathbb{P}^{S}_{\pi^\A}(\bb)$ for each $\pi^\A\in\mathbb{P}(\A)$ and, since $\mathbb{P}^{fs}(\bb)$ is dense in $\mathbb{P}(\bb),$ then $\cap_{\pi^\A\in\mathbb{P}(\A)}\mathbb{P}^{S}_{\pi^\A}(\bb)$ is dense in $\mathbb{P}(\bb).$}
\end{remark}

\begin{remark}\label{rem:new2}
{\rm
Assumption~(v) for the game $\{\A,\bb,c\}$ implies that ${\hat{c}}^\oplus(a,\pi^\bb)<+\infty$ for each $a\in\A$ and $\pi^\bb\in \mathbb{P}(\bb).$ Thus,
$\mathbb{P}^{fs}(\A)\subset \mathbb{P}^{S}_{\pi^\bb}(\A)$ for each $\pi^\bb\in\mathbb{P}(\bb)$ and, since $\mathbb{P}^{fs}(\A)$ is dense in $\mathbb{P}(\A),$ then $\cap_{\pi^\bb\in\mathbb{P}(\bb)}\mathbb{P}^{S}_{\pi^\bb}(\A)$ is dense in $\mathbb{P}(\A).$}
\end{remark}

The set of all strategies for each player is partitioned into the sets of \emph{safe} strategies $\mathbb{P}^{S}(\A)$ and $\mathbb{P}^{S}(\bb)$  (strategies, for which the expected payoff is well-defined for all strategies played by another player) and \emph{unsafe} strategies $\mathbb{P}^{U}(\A)$ and $\mathbb{P}^{U}(\bb):$ 
\[
\begin{aligned}
\mathbb{P}^{S}(\A)&:=\{\pi^\A\in\mathbb{P}(\A)\,:\, \mathbb{P}^{S}_{\pi^\A}(\bb)=\mathbb{P}(\bb)\},\\ &\mathbb{P}^{U}(\A):=\{\pi^\A\in\mathbb{P}(\A)\,:\, \mathbb{P}^{S}_{\pi^\A}(\bb)\ne\mathbb{P}(\bb)\};\\
\mathbb{P}^{S}(\bb)&:=\{\pi^\bb\in\mathbb{P}(\bb)\,:\, \mathbb{P}^{S}_{\pi^\bb}(\A)=\mathbb{P}(\A)\},\\ &\mathbb{P}^{U}(\bb):=\{\pi^\bb\in\mathbb{P}(\bb)\,:\, \mathbb{P}^{S}_{\pi^\bb}(\A)\ne\mathbb{P}(\A)\}.
\end{aligned}
\]

\begin{remark}\label{rem:uncer}{\rm
We note that $\mathbb{P}(\A)=\mathbb{P}^{S}(\A)\cup \mathbb{P}^{U}(\A),$ $\mathbb{P}(\bb)=\mathbb{P}^{S}(\bb)\cup \mathbb{P}^{U}(\bb),$ $\mathbb{P}^{S}(\A)\cap \mathbb{P}^{U}(\A)=\emptyset,$ and $\mathbb{P}^{S}(\bb)\cap \mathbb{P}^{U}(\bb)=\emptyset.$ Moreover,
$\mathbb{P}^{fs}(\A)\subset \mathbb{P}^{S}(\A)$ (see assumption (iv) in Definition~\ref{defi:game} of the game $\{\A,\bb,c\}$ and Remark~\ref{rem:new1}) and $\mathbb{P}^{fs}(\bb)\subset \mathbb{P}^{S}(\bb)$ (see assumption (v) in Definition~\ref{defi:game} and Remark~\ref{rem:new2}). Therefore, $\mathbb{P}^{S}(\A)$ is dense in $\mathbb{P}(\A)$ and
$\mathbb{P}^{S}(\bb)$ is dense in $\mathbb{P}(\bb).$ 
}
\end{remark}
\begin{remark}\label{rem:equivEFAB}{\rm 
Observe that $\mathbb{P}^S(\A)=\mathbb{P}(\A)$ if and only if $\hat{c}(\pi^\A,\pi^\bb)$ is well-defined for all pairs $(\pi^\A,\pi^\bb)\in\mathbb{P}(\A)\times\mathbb{P}(\bb).$ Therefore, the following five claims are equivalent: (i) $\mathbb{P}^S(\A)=\mathbb{P}(\A),$ (ii) $\mathbb{P}^U(\A)=\emptyset,$ (iii) $\mathbb{P}^S(\bb)=\mathbb{P}(\bb),$ (iv) $\mathbb{P}^U(\bb)=\emptyset,$ (v) $\hat{c}(\pi^\A,\pi^\bb)$ is well-defined for all pairs $(\pi^\A,\pi^\bb)\in\mathbb{P}(\A)\times\mathbb{P}(\bb).$
}
\end{remark}

Let us introduce the following notations:
\begin{equation}\label{eq:maxmin}
\begin{aligned}
&{\hat{c}}^{\sharp}(\pi^\A):=\sup\limits_{b\in \bb} {\hat{c}}(\pi^\A,b), &\mathbb{P}^\sharp_\alpha(\A):=\{\pi_*^\A\in \mathbb{P}(\A)\,:\, {\hat{c}}^\sharp(\pi^\A_*)\le \alpha\},\\
&{\hat{c}}^{\flat}(\pi^\bb):=\inf\limits_{a\in \A } {\hat{c}}(a,\pi^\bb),
&\mathbb{P}^\flat_\beta(\bb):=\{\pi_*^\bb\in \mathbb{P}(\bb)\,:\, {\hat{c}}^\flat(\pi_*^\bb)\ge \beta\},
\end{aligned}
\end{equation}
for each $\pi^\A\in \mathbb{P}(\A),$ $\pi^\bb\in \mathbb{P}(\bb),$ $\alpha,\beta \in\R.$ Remarks~\ref{rem:new1} and \ref{rem:new2} imply respectively that ${\hat{c}}^{\sharp}(\pi^\A)>-\infty$ for all $\pi^\A\in\mathbb{P}(\A)$ and ${\hat{c}}^{\flat}(\pi^\bb)<+\infty$ for all $\pi^\bb\in\mathbb{P}(\bb).$

\begin{theorem}\label{cor:mainI}
Let  $\{\A,\bb, c\}$ be  a two-person zero-sum game  introduced in Definition~\ref{defi:game} and 
$(\pi^\A,\pi^\bb)\in\mathbb{P}(\A)\times\mathbb{P}(\bb).$  Then
the following two equalities hold:
\begin{equation}\label{eq:prop:mainIa}
{\hat{c}}^{\sharp}(\pi^\A)=\sup_{\pi_*^\bb\in\mathbb{P}^S_{\pi^\A}(\bb)} {\hat{c}}(\pi^\A,\pi_*^\bb ),
\end{equation}
\begin{equation}\label{eq:prop:mainIb}
{\hat{c}}^{\flat}(\pi^\bb)=\inf_{\pi_*^\A\in \mathbb{P}^S_{\pi^\bb}(\A)} {\hat{c}}(\pi_*^\A,\pi^\bb),
\end{equation}
where ${\hat{c}}^\sharp$ and ${\hat{c}}^\flat$ are defined in (\ref{eq:maxmin}).
\end{theorem}
\begin{proof}
It is sufficient to establish equality (\ref{eq:prop:mainIa}) for each $\pi^\A\in\mathbb{P}(\A).$ Indeed,
 equality (\ref{eq:prop:mainIa}), being applied to the game $\{\bb,\A,-c^{\A\leftrightarrow\bb}\},$ where the function $c^{\A\leftrightarrow\bb}$ is defined in Remark~\ref{rem:sim}, implies equality (\ref{eq:prop:mainIb}). 

Let us prove that equality (\ref{eq:prop:mainIa}) holds for each $\pi^\A\in \mathbb{P}(\A).$ Fix an arbitrary $\pi^\A\in\mathbb{P}(\A).$

According to Remark~\ref{rem:new1}, the expected
payoff ${\hat{c}}(\pi^\A,b)$ to  Player II is well-defined for each $b\in\bb.$ Then the inequality \[\sup_{\pi_*^\bb\in \mathbb{P}^S_{\pi^\A}(\bb)} {\hat{c}}(\pi^\A,\pi_*^\bb)\ge\sup_{b\in \bb} {\hat{c}}(\pi^\A,b)={\hat{c}}^{\sharp}(\pi^\A)\] holds because
each pure strategy for Player~II
can be interpreted as the mixed strategy concentrated in a point.
Now let us prove that
\[
{\hat{c}}^{\sharp}(\pi^\A)\le\sup_{\pi_*^\bb\in\mathbb{P}^S_{\pi^\A}(\bb)} {\hat{c}}(\pi^\A,\pi_*^\bb ).
\]
If $\sup_{b\in \bb} {\hat{c}}(\pi^\A,b)=+\infty,$ then the inequality
\begin{equation}\label{eq:th25}
\sup_{\pi_*^\bb\in \mathbb{P}^S_{\pi^\A}(\bb)} {\hat{c}}(\pi^\A,\pi_*^\bb)\le\sup_{b\in \bb} {\hat{c}}(\pi^\A,b)
\end{equation}
obviously holds. Let $\sup_{b\in \bb} {\hat{c}}(\pi^\A,b)<+\infty.$
Inequality (\ref{eq:th25}) holds if and only if
\begin{equation}\label{eq:the251}
{\hat{c}}(\pi^\A,\pi_*^\bb)\le\sup_{b\in \bb} {\hat{c}}(\pi^\A,b)
\end{equation}
for each $\pi_*^\bb\in \mathbb{P}^S_{\pi^\A}(\bb).$
The rest of the proof establishes inequality (\ref{eq:the251}).

Let us fix an arbitrary $\pi_*^\bb\in \mathbb{P}^S_{\pi^\A}(\bb).$ Since either ${\hat{c}}^\ominus(\pi^\A,\pi_*^\bb)>-\infty$ or ${\hat{c}}^\oplus(\pi^\A,\pi_*^\bb)<+\infty,$ we have that the Fubini-Tonelli theorem implies
\[
{\hat{c}}(\pi^\A,\pi_*^\bb)=\int_\bb {\hat{c}}(\pi^\A,b)\pi_*^\bb(db),
\]
which implies (\ref{eq:the251}). Inequality (\ref{eq:th25}) is proved.
\qed \end{proof}

\begin{remark}\label{rem:main2}
{\rm According to (\ref{eq:maxmin}) and assumptions (iv) and (v) in  Definition~\ref{defi:game} of the game $\{\A,\bb,c\}$  (see also Remarks~\ref{rem:new1} and \ref{rem:new2} and Theorem~\ref{cor:mainI}), 
the  inequality
\begin{equation}\label{eq:g}
{\hat{c}}^{\flat}(\pi^\bb)\le {\hat{c}}^{\sharp}(\pi^\A)
\end{equation}
holds for all $\pi^\A\in\mathbb{P}(\A)$ and for all $\pi^\bb\in\mathbb{P}^S(\bb).$  Indeed, for $\pi^\bb\in\mathbb{P}^S(\bb)$ and for $\pi^\A\in\mathbb{P}(\A),$
\[
{\hat{c}}^{\flat}(\pi^\bb)=\inf_{\pi_*^\A\in \mathbb{P}(\A)} {\hat{c}}(\pi_*^\A,\pi^\bb)\le {\hat{c}}(\pi^\A,\pi^\bb)\le \sup_{\pi_*^\bb\in\mathbb{P}^S_{\pi^\A}(\bb)} {\hat{c}}(\pi^\A,\pi_*^\bb )
= {\hat{c}}^{\sharp}(\pi^\A).
\]
Since it is not clear whether inequality \eqref{eq:g} holds for $\pi^\bb\in\mathbb{P}^U(\bb),$ the following definition introduces the lopsided value (the value in the asymmetric form).
}
\end{remark}

\begin{definition}\label{Def5.8}
If the   equality
\begin{equation}\label{eq:valval}
\sup_{\pi^\bb\in\mathbb{P}^S(\bb)}{\hat{c}}^{\flat}(\pi^\bb)=\inf_{\pi^\A\in\mathbb{P}(\A)} {\hat{c}}^{\sharp}(\pi^\A)\ (:=v)
\end{equation}
holds, then we say that $v$ is the \textit{lopsided value of the game} $\{\A,\bb,c\}.$ 
\end{definition}
\begin{remark}\label{rem:lopsided}
{\rm
The lopsided value coincides with the classical definition of the value if $\mathbb{P}^S(\bb)=\mathbb{P}(\bb).$  In this case, \eqref{eq:valval} becomes
\begin{equation}\label{rem:lopsidedclas}
\sup_{\pi^\bb\in\mathbb{P}(\bb)}{\hat{c}}^{\flat}(\pi^\bb)=\inf_{\pi^\A\in\mathbb{P}(\A)} {\hat{c}}^{\sharp}(\pi^\A).
\end{equation}
For example, if  $c$ is bounded either from below or from  above on $\A\times\bb,$ then $\mathbb{P}^S(\bb)=\mathbb{P}(\bb).$
If \eqref{rem:lopsidedclas} holds, we shall omit the term ``lopsided.''
}
\end{remark}

\begin{remark}\label{rem:lopsidedassym}
Infsup equality \eqref{eq:valval} is asymmetric.  The main obstacle for writing it in the symmetric form \eqref{rem:lopsidedclas} is that it is not clear why inequality \eqref{eq:g} holds for  all $\pi^\A\in\mathbb{P}(\A)$ and  $\pi^\bb\in\mathbb{P}(\bb).$ Equality \eqref{eq:valval} can be linked to general forms of infsup equalities, which are asymmetric; see Proposition I.1.9 in Mertens et al.~\cite{MSZ}.
This is the reason why we use the term lopsided value.
\end{remark}

In addition to the sets $\mathbb{P}^\sharp_{\alpha}(\A)$ and $\mathbb{P}^\flat_{\beta}(\bb)$ defined in \eqref{eq:maxmin}, let us introduce 
\[
\begin{aligned}
&\mathbb{P}^\sharp_{<\alpha}(\A):=\{\pi^\A\in \mathbb{P}(\A)\,:\, {\hat{c}}^\sharp(\pi^\A)< \alpha\},\quad \alpha\in\R,\\
&\mathbb{P}^\flat_{>\beta}(\bb):=\{\pi^\bb\in \mathbb{P}(\bb)\,:\, {\hat{c}}^\flat(\pi^\bb)> \beta\},\quad \beta\in\R.
\end{aligned}
\]

\begin{lemma}\label{lem:con}
Let $\{\A,\bb, c\}$ be a  two-person zero-sum game  introduced in Definition~\ref{defi:game}. Then the following statements hold:
\begin{itemize}
\item[{\rm(a)}] the function ${\hat{c}}^{\sharp}$ is convex on $\mathbb{P}(\A);$
\item[{\rm(b)}] the function ${\hat{c}}^{\flat}$ is concave on $\mathbb{P}(\bb);$
\item[{\rm(c)}] the sets $\mathbb{P}^\sharp_\alpha(\A),$ $\mathbb{P}^\sharp_{<\alpha}(\A),$ $\mathbb{P}^\flat_\beta(\bb),$ and $\mathbb{P}^\flat_{>\beta}(\bb)$ are convex for all $\alpha,\beta\in\R;$
\end{itemize}
\end{lemma}

\begin{proof}

Let us prove statement (a). Indeed, let $\pi_1^\A,\pi_2^\A\in \mathbb{P}(\A)$ and $\alpha\in (0,1).$ If either ${\hat{c}}^{\sharp}(\pi_1^\A)=+\infty$ or  ${\hat{c}}^{\sharp}(\pi_2^\A)=+\infty,$ then ${\hat{c}}^{\sharp}(\alpha\pi_1^\A+(1-\alpha)\pi_2^\A)\le \alpha{\hat{c}}^{\sharp}(\pi_1^\A)+(1-\alpha){\hat{c}}^{\sharp}(\pi_2^\A).$ Otherwise,  $\pi_1^\A,\pi_2^\A\in \mathbb{P}_{<+\infty}^\sharp(\A)$ and
\begin{equation}\label{eq:eq}
\begin{aligned}
&\alpha{\hat{c}}^{\sharp}(\pi_1^\A)+(1-\alpha){\hat{c}}^{\sharp}(\pi_2^\A)= \alpha\sup_{b\in\bb} \hat c(\pi_1^\A,b)+(1-\alpha)\sup_{b\in\bb} \hat c(\pi_2^\A,b)\\ &\ge \sup_{b\in\bb} \hat c(\alpha\pi_1^\A+(1-\alpha)\pi_2^\A,b)={\hat{c}}^{\sharp}(\alpha\pi_1^\A+(1-\alpha)\pi_2^\A).
\end{aligned}
\end{equation}
Since $\pi_1^\A,\pi_2^\A\in \mathbb{P}(\A)$ and $\alpha\in (0,1)$ are arbitrary, then (\ref{eq:eq}) implies that the worst-loss function ${\hat{c}}^{\sharp}$ is convex on $\mathbb{P}(\A).$ Statement (a) is proved.

Statement (b) follows from statement (a) applied to
$\{\bb,\A,-c^{\A\leftrightarrow\bb}\},$ where $c^{\A\leftrightarrow\bb}(b,a):=c(a,b)$ for each $a\in\A$ and $b\in\bb.$ Statement (c) follows from statements (a) and (b).
\qed \end{proof}

\subsection{The Existence of a Lopsided Value}\label{sub:exval}

The following Theorem~\ref{th:exval} provides sufficient conditions for the existence of a lopsided value for a   two-person zero-sum game with
possibly noncompact action sets and unbounded payoffs and describes the property of the solution set for one of the player under these conditions.  For well-defined payoff
functions, the proof of the existence of the value is usually based on Sion's theorem (Mertens et al.~\cite[Theorem I.1.1] {MSZ}) that requires that at least one   of the decision sets is compact.  In our situation, both decision sets may not be compact.  In addition, the payoff function $c$ may be unbounded above and below, and therefore the payoff function $\hat c$ may be undefined for some pairs of mixed strategies.  Because of these reasons, our proof of the existence of the lopsided value does not use Sion's theorem. In general,  a game on the unit square with bounded measurable payoffs  may not have  a value; see Yanovskaya~\cite[p. 527]{Yan}, and the references to counterexamples by Ville, by Wald, and by Sion and by Wolfe cited there. Therefore, some conditions on continuity of  payoff functions are needed, and Theorem~\ref{th:exval} requires mild assumptions (i) and (ii). 

\begin{theorem}\label{th:exval}
Let a  two-person zero-sum game $\{\A,\bb, c \}$  introduced in Definition~\ref{defi:game} satisfy the following assumptions:
\begin{itemize}
\item[{\rm(i)}] for each $b\in \bb$ the function $a\mapsto c(a,b)$ is lower semi-continuous;
\item[{\rm(ii)}] there exists $b_0\in \bb$ such that the function $a\mapsto c(a,b_0)$ is inf-compact on $\A.$
\end{itemize}
Then the game $\{\A,\bb, c \}$ has a lopsided value $v,$ that is, equality   \eqref{eq:valval} holds,  and
\[\sup_{\pi^\bb\in\mathbb{P}^S(\bb)}{\hat{c}}^{\flat}(\pi^\bb)=\sup_{\pi^\bb\in\mathbb{P}^{fs}(\bb)}{\hat{c}}^{\flat}(\pi^\bb).
\]
%
Moreover, the set $\mathbb{P}^\sharp_{v}(\A)$ is a
nonempty convex compact subset of $\mathbb{P}(\A).$%
\end{theorem}
Let $\F(S)$ be the family of all finite subsets of a set $S.$ The proof of Theorem~\ref{th:exval} uses the following theorem.
\begin{theorem}{\rm(Aubin and Ekeland~\cite[Theorem~6.2.2]{ObEk})}\label{teor:obek}
Let $A$ and $B$ be nonempty convex subsets of vector spaces and $f:A \times B\mapsto \mathbb{R}$ be a  function such that $a\mapsto f(a,b)$ is convex for each $b\in B$ and $b\mapsto f(a,b)$ is concave for each $a\in A.$ Then 
\begin{equation}\label{eq:obaux}
 \sup_{b\in B} \inf_{a\in A } f(a,b)= \sup_{F\in \F (B) } \inf_{a\in A }\max_{b\in F} f (a,b).
\end{equation}
\end{theorem}
%

\begin{proof}{\it of Theorem~\ref{th:exval}}
%
%
Observe that the following statements hold:
\begin{itemize}
\item[{\rm($i_1$)}] the sets $\mathbb{P}_{<+\infty}^\sharp(\A)$ and $\mathbb{P}^{fs}(\bb)$ are nonempty and convex;
 \item[{\rm($i_2$)}] the function ${\hat{c}}:\mathbb{P}(\A)\times \mathbb{P}^{fs}(\bb)\mapsto\mathbb{R}\cup\{+\infty\}$ is well-defined and affine in each variable;
\item[{\rm($i_3$)}] the function ${\hat{c}}(\,\cdot\,,\pi^\bb):\mathbb{P}(\A)\mapsto \mathbb{R}\cup\{+\infty\}$ is lower semi-continuous for each $\pi^\bb\in \mathbb{P}^{fs}(\bb);$
\item[{\rm($i_4$)}] the function ${\hat{c}}(\,\cdot\,,b_0):\mathbb{P}(\A)\mapsto \mathbb{R}\cup\{+\infty\}$ is inf-compact on $\mathbb{P}(\A);$
\item[{\rm($i_5$)}] the function ${\hat{c}}(\cdot,\cdot)$ takes finite values on $\mathbb{P}_{<+\infty}^\sharp(\A)\times \mathbb{P}^{fs}(\bb).$
\end{itemize}

Let us prove statements {\rm($i_1$)}--{\rm($i_5$)}.

{\rm($i_1$)} 
According to Remark~\ref{rem:new2},
${\hat c}^\sharp(\pi^\A)<+\infty$ for some $\pi^\A\in\mathbb{P}(\A).$   Thus   the set $\mathbb{P}_{<+\infty}^\sharp(\A)$ is  not empty. Lemma~\ref{lem:con}(c) implies that the set $\mathbb{P}_{<+\infty}^\sharp(\A)$ is convex. The set $\mathbb{P}^{fs}(\bb)$ is not empty since the set of pure strategies for Player~II is not empty and each pure strategy for Player~II belongs to $\mathbb{P}^{fs}(\bb).$ The set $\mathbb{P}^{fs}(\bb)$ is convex because a convex combination of two probability measures on $\bb$ with finite supports is a probability measure on $\bb$ with a finite support. Statement {\rm($i_1$)} is proved.

{\rm($i_2$)} Let $\pi^\A\in  \mathbb{P}(\A)$  and $\pi^\bb\in \mathbb{P}^{fs}(\bb).$ The definition of $\mathbb{P}^{fs}(\bb)$ implies the existence of $M=1,2,\ldots,$ $\{\beta^{(m)}\}_{m=1,2,\ldots,M}\subset [0,1],$ and $\{b^{(m)}\}_{m=1,2,\ldots,M}\subset \bb$ such  that $\beta^{(1)}+\beta^{(2)}+\ldots+\beta^{(M)}=1$ and $\pi^\bb(B)=\beta^{(1)}{\bf I}\{b^{(1)}\in B\}+\beta^{(2)}{\bf I}\{b^{(2)}\in B\}+\ldots+\beta^{(M)}{\bf I}\{b^{(M)}\in B\}$ for each $B\in\B(\bb),$ where ${\bf I}\{b\in B\}=1$ whenever  $b\in B$ and ${\bf I}\{b\in B\}=0$ otherwise. Since the function $a\mapsto c(a,b)$ is bounded from below on $\A$
for each $b\in\bb,$ 
\begin{equation}\label{eq:eq1}
\begin{aligned}
{\hat{c}}^\ominus(\pi^\A,\pi^\bb)&= \int_\A \left(\sum_{m=1}^M\beta^{(m)} c^-(a,b^{(m)})\right)\pi^\A(da)\ge \\ &\sum_{m=1}^M\beta^{(m)}\inf_{a\in\A}c^-(a,b^{(m)})>-\infty,
\end{aligned}
\end{equation}
which implies that ${\hat{c}}(\pi^\A,\pi^\bb)$ is well-defined for all  $\pi^\A\in \mathbb{P}(\A)$  and for all $\pi^\bb\in \mathbb{P}^{fs}(\bb).$ This function is affine in each variable on $ \mathbb{P}(\A)\times \mathbb{P}^{fs}(\bb)$ because of the basic properties of the Lebesgue integral.
Statement {\rm($i_2$)} is proved.
%

{\rm($i_3$)} Let us fix an arbitrary $\pi^\bb\in \mathbb{P}^{fs}(\bb).$
As shown in the proof of {\rm($i_2$)}, there exist $M=1,2,\ldots,$ $\{\beta^{(m)}\}_{m=1,2,\ldots,M}\subset [0,1],$ and $\{b^{(m)}\}_{m=1,2,\ldots,M}\subset \bb$ such  that $\beta^{(1)}+\beta^{(2)}+\ldots+\beta^{(M)}=1,$ and $\pi^\bb(B)=\beta^{(1)}{\bf I}\{b^{(1)}\in B\}+\beta^{(2)}{\bf I}\{b^{(2)}\in B\}+\ldots+\beta^{(M)}{\bf I}\{b^{(M)}\in B\}$ for each $B\in\B(\bb).$  Since ${\hat{c}}(\pi^\A,\pi^\bb)=\beta^{(1)}{\hat{c}}(\pi^\A,b^{(1)})+\beta^{(2)}{\hat{c}}(\pi^\A,b^{(2)})+\ldots+\beta^{(M)}{\hat{c}}(\pi^\A,b^{(M)})$ for each $\pi^\A\in\mathbb{P}(\A),$ it is sufficient to prove that the function  ${\hat{c}}(\,\cdot\,,b):\mathbb{P}(\A)\mapsto \mathbb{R}\cup\{+\infty\}$ is lower semi-continuous for each $b\in \bb$ because a convex combination of a finite number of bounded  below lower semi-continuous functions is lower semi-continuous. Lemma~\ref{lem:lsc}, being applied  to $\mathbb{S}_1=\{b\},$ $\mathbb{S}_2=\A,$ and $f(s_1,s_2)=c(s_2,s_1),$ $(s_1,s_2)\in\mathbb{S}_1\times\mathbb{S}_2,$  implies that the function  ${\hat{c}}(\,\cdot\,,b):\mathbb{P}(\A)\mapsto \mathbb{R}\cup\{+\infty\}$ is lower semi-continuous for each $b\in \bb.$ Statement {\rm($i_3$)} is proved.


($i_4$) Assumption (i) and Theorem~\ref{th:inf-comp}, being applied to $\mathbb{S}_1=\{b\},$ $\mathbb{S}_2=\A,$  and $f(s_1,s_2)=c(s_2,s_1),$ $(s_1,s_2)\in\mathbb{S}_1\times\mathbb{S}_2,$ imply that
the function ${\hat{c}}(\,\cdot\,,b_0):\mathbb{P}(\A)\mapsto \mathbb{R}\cup\{+\infty\}$ is inf-compact on $\mathbb{P}(\A).$ Statement {\rm($i_4$)} is proved.

  ($i_5$) Let $\pi^\A\in \mathbb{P}_{<+\infty}^\sharp(\A)$  and $\pi^\bb\in \mathbb{P}^{fs}(\bb).$ Note that
\begin{equation}\label{eq:eq2}
{\hat{c}}(\pi^\A,\pi^\bb)\le{\hat{c}}^{\sharp}(\pi^\A)<+\infty,
\end{equation}
for all $\pi^\A\in \mathbb{P}_{<+\infty}^\sharp(\A)$  and for all $\pi^\bb\in \mathbb{P}^{fs}(\bb),$ where the first inequality follows from 
(\ref{eq:prop:mainIa}) and Remark~\ref{rem:new1}. The second one 
follows from 
$\pi^\mathbb{A}\in\mathbb{P}_{<+\infty}^\sharp(\A).$

Inequalities (\ref{eq:eq1}) and (\ref{eq:eq2}) imply that the function ${\hat{c}}(\cdot,\cdot)$ takes finite values on $\mathbb{P}_{<+\infty}^\sharp(\A)\times \mathbb{P}^{fs}(\bb).$ Statement {\rm($i_5$)} is proved.


Let us prove 
 equality (\ref{eq:valval}). 
In view of inequality (\ref{eq:g}), it is sufficient to prove that
\begin{equation}\label{eq:ob0}
\inf\limits_{\pi^\A\in \mathbb{P}(\A)} {\hat{c}}^{\sharp}(\pi^\A)\le\sup\limits_{\pi^\bb\in \mathbb{P}^S(\bb)}{\hat{c}}^{\flat}(\pi^\bb).
\end{equation}
%
We denote the left-hand side of inequality (\ref{eq:ob0}) by $v^\sharp$ and the right-hand side of inequality (\ref{eq:ob0}) by $v^\flat.$
Since $\mathbb{P}^{fs}(\bb)\subset \mathbb{P}^S(\bb)$  (see Remark~\ref{rem:uncer}), 
\begin{equation}\label{eq:ob1}
 \sup_{\pi^\bb\in\mathbb{P}^{fs}(\bb)} {\hat{c}}^\flat(\pi^\bb)\le v^\flat.
\end{equation}
Since $\mathbb{P}^{fs}(\bb)\subset \mathbb{P}^S(\bb)$, formulae (\ref{eq:maxmin}) and (\ref{eq:prop:mainIb}) imply that 
for each $\pi^\bb\in\mathbb{P}^{fs}(\bb)$
\begin{equation}\label{eq:ob2}
  {\hat{c}}^\flat(\pi^\bb)=\inf_{a\in \A} {\hat{c}}(a,\pi^\bb)=\inf_{\pi^\A\in \mathbb{P}(\A)}{\hat{c}}(\pi^\A,\pi^\bb),
\end{equation}
where the second equality follows from $\mathbb{P}_{\pi^\bb}^S(\A)=\mathbb{P}(\A)$ since $\pi^\bb\in\mathbb{P}^S(\bb).$
%
In view of assumption (v) from Definition~\ref{defi:game}, each pure strategy of Player I belongs to $\mathbb{P}_{<+\infty}^\sharp(\A)\subset\mathbb{P}(\A).$ 
Therefore,
 (\ref{eq:ob2}) implies
\begin{equation}\label{eq:ob3EF}
  {\hat{c}}^\flat(\pi^\bb)=\inf_{\pi^\A\in \mathbb{P}_{<+\infty}^\sharp(\A)}{\hat{c}}(\pi^\A,\pi^\bb),
\end{equation}
for each $\pi^\bb\in\mathbb{P}^{fs}(\bb).$
Inequality (\ref{eq:ob1}) and equality \eqref{eq:ob3EF} imply
\begin{equation}\label{eq:ob3}
 \sup_{\pi^\bb\in\mathbb{P}^{fs}(\bb)} \inf_{\pi^\A\in \mathbb{P}_{<+\infty}^\sharp(\A)}{\hat{c}}(\pi^\A,\pi^\bb)\le v^\flat.
\end{equation}
%
%
%
%
%
In view of  properties  ($i_1$), ($i_2$), and ($i_5$),
Theorem~\ref{teor:obek}, with $A=\mathbb{P}_{<+\infty}^\sharp(\A),$ $B=\mathbb{P}^{fs}(\bb),$ and $f={\hat{c}},$ implies
\begin{equation}\label{eq:ob4}
 \sup_{\pi^\bb\in\mathbb{P}^{fs}(\bb)} \inf_{\pi^\A\in \mathbb{P}_{<+\infty}^\sharp(\A)}{\hat{c}}(\pi^\A,\pi^\bb)= \sup_{F\in \F(\mathbb{P}^{fs}(\bb))} \inf_{\pi^\A\in \mathbb{P}_{<+\infty}^\sharp(\A)}\max_{\pi^\bb\in F}{\hat{c}}(\pi^\A,\pi^\bb).
\end{equation}

Let $\F_0(\mathbb{P}^{fs}(\bb))$ denote the family of all finite subsets of $\mathbb{P}^{fs}(\bb)$ containing the pure strategy of Player~II concentrated at the point $b_0\in\bb,$ whose existence is stated in assumption~(ii). 
Since $\mathbb{P}_{<+\infty}^\sharp(\A)\subset\mathbb{P}(\A)$ and $\F_0(\mathbb{P}^{fs}(\bb))\subset \F(\mathbb{P}^{fs}(\bb)),$ 
\begin{equation}\label{eq:ob5}
v^*:=\sup_{F\in \F_0(\mathbb{P}^{fs}(\bb))} \inf_{\pi^\A\in \mathbb{P}(\A)}\max_{\pi^\bb\in F}{\hat{c}}(\pi^\A,\pi^\bb)\le\sup_{F\in \F(\mathbb{P}^{fs}(\bb))} \inf_{\pi^\A\in \mathbb{P}_{<+\infty}^\sharp(\A)}\max_{\pi^\bb\in F}{\hat{c}}(\pi^\A,\pi^\bb).
\end{equation}
Formulae \eqref{eq:ob3}--\eqref{eq:ob5} imply $v^*\le v^\flat.$  Thus,
if
\begin{equation}\label{eq:ob6}
v^\sharp\le v^*,   
\end{equation}
then  inequality (\ref{eq:ob0}) holds. Recall that inequality (\ref{eq:ob0}) implies 
 equality  (\ref{eq:valval}). 

Let us prove  (\ref{eq:ob6}). 
Statements ($i_3$) and ($i_4$) imply that the function\\ $\max_{\pi^\bb\in F}{\hat{c}}(\,\cdot\,,\pi^\bb)$ is inf-compact on $\mathbb{P}(\A)$ for each $F\in \F_0(\mathbb{P}^{fs}(\bb)).$ Therefore, for each $F\in \F_0(\mathbb{P}^{fs}(\bb))$ there exists $\pi^\A_F\in \mathbb{P}(\A)$ such that \[\pi^\A_F=\mbox{arg\,min}_{\pi^\A\in\mathbb{P}(\A)}\max_{\pi^\bb\in F}{\hat{c}}(\pi^\A,\pi^\bb).\] The definition of $v^*$ given in  (\ref{eq:ob5}) implies that $\pi^\A_F\in \cap_{\pi^\bb\in F}\mathcal{D}_{{\hat{c}}(\,\cdot\,,\pi^\bb)}(v^*)$  for each $F\in \F_0(\mathbb{P}^{fs}(\bb)).$ Thus, for each $F\in \F_0(\mathbb{P}^{fs}(\bb)),$
\begin{equation}\label{eq:ob7}
\cap_{\pi^\bb\in F}\mathcal{D}_{{\hat{c}}(\,\cdot\,,\pi^\bb)}(v^*)\ne \emptyset.
\end{equation}

Statement ($i_3$) and Remark~\ref{rem:tildelsc} imply that the set $\mathcal{D}_{{\hat{c}}(\,\cdot\,,\pi^\bb)}(v^*)$ is closed for each $\pi^\bb\in\mathbb{P}^{fs}(\bb).$
Statement ($i_4$) implies that the set  $\mathcal{D}_{{\hat{c}}(\,\cdot\,,b_0)}(v^*)$ is compact. 
As follows from  (\ref{eq:ob7}), the collection $\{\mathcal{D}_{{\hat{c}}(\,\cdot\,,\pi^\bb)}(v^*)\cap \mathcal{D}_{{\hat{c}}(\,\cdot\,,b_0)}(v^*)\}_{\pi^\bb\in\mathbb{P}^{fs}(\bb)}$ of closed subsets of the compact set $\mathcal{D}_{{\hat{c}}(\,\cdot\,,b_0)}(v^*)$ satisfies the finite intersection property. Therefore, this collection 
has a nonempty intersection, that is, there exists $\pi_*^\A\in \mathbb{P}(\A)$ such that $\pi_*^\A\in \cap_{ \pi^\bb\in\mathbb{P}^{fs}(\bb)}\mathcal{D}_{{\hat{c}}(\,\cdot\,,\pi^\bb)}(v^*);$ see e.g., Reed and Simon \cite[p.~98]{SR}. 
Thus   ${\hat{c}}(\pi_*^\A,\pi^\bb)\le v^*$ for all $\pi^\bb\in\mathbb{P}^{fs}(\bb),$ and  therefore
\begin{equation}\label{eq:ob8}
\sup_{\pi^\bb\in\mathbb{P}^{fs}(\bb)}{\hat{c}}(\pi_*^\A,\pi^\bb)\le v^*.
\end{equation}
We note that 
\begin{equation}\label{eq:ob9}
{\hat{c}}^\sharp(\pi^\A_*)=\sup_{b\in\bb}\hat{c}(\pi^\A_*,b) \le \sup_{\pi^\bb\in\mathbb{P}^{fs}(\bb)}{\hat{c}}(\pi_*^\A,\pi^\bb),
\end{equation}
where the equality is the first definition in \eqref{eq:maxmin} and the inequality  holds because each pure strategy of Player~II  
belongs to $\mathbb{P}^{fs}(\bb).$ 

Inequalities (\ref{eq:ob8}), (\ref{eq:ob9}) and the definition of $v^\sharp$ 
imply inequality (\ref{eq:ob6}), which implies inequality (\ref{eq:ob0}). Thus,   equality  (\ref{eq:valval}) holds.

Let us prove that the set $\mathbb{P}^\sharp_{v}(\A)$
 is a nonempty convex compact subset of $\mathbb{P}(\A).$ The nonemptyness of 
 the set $\mathbb{P}^\sharp_{v}(\A)$ follows from (\ref{eq:ob8}) and (\ref{eq:ob9}) because $v^*=v^\sharp=v,$ where $v$ is introduced in Definition~\ref{Def5.8}. 
As follows from the definition of $\mathbb{P}^\sharp_{v}(\A)$ in \eqref{eq:maxmin},
\begin{equation}\label{endofproofT18EF}
\mathbb{P}^\sharp_{v}(\A)=\cap_{b\in \bb}\mathcal{D}_{{\hat{c}}(\,\cdot\,, b)}(v).
\end{equation}
According to properties ($i_2$)--($i_5$), 
 the set $\mathcal{D}_{{\hat{c}}(\,\cdot\,, {b_0})}(v)$ is a convex compact subset of $\mathbb{P}(\A)$ and the set
   $\mathcal{D}_{{\hat{c}}(\,\cdot\,, \pi^{\bb})}(v)$ is a convex closed subset of $\mathbb{P}(\A)$ for each $\pi^\bb\in \mathbb{P}^{fs}(\bb).$
   In particular, the set
   $\mathcal{D}_{{\hat{c}}(\,\cdot\,, b)}(v)$ is a convex closed subset of $\mathbb{P}(\A)$ for each $b\in \bb.$
   Therefore, formula~\eqref{endofproofT18EF} implies that $\mathbb{P}^\sharp_{v}(\A)$
 is a nonempty convex compact subset of $\mathbb{P}(\A).$

 To finish the proof we note that equalities (\ref{eq:ob3EF}) and (\ref{eq:ob4}) and inequalities (\ref{eq:ob5}) and (\ref{eq:ob6}) imply
\[
\inf_{\pi^\A\in\mathbb{P}(\A)} {\hat{c}}^{\sharp}(\pi^\A)\le \sup_{\pi^\bb\in\mathbb{P}^{fs}(\bb)}{\hat{c}}^{\flat}(\pi^\bb).
\]
Therefore, the equality
\[
\sup_{\pi^\bb\in\mathbb{P}^S(\bb)}{\hat{c}}^{\flat}(\pi^\bb)=\sup_{\pi^\bb\in\mathbb{P}^{fs}(\bb)}{\hat{c}}^{\flat}(\pi^\bb).
\]
follows from \eqref{eq:valval} and (\ref{eq:ob1}).
\qed \end{proof}

\begin{corollary}\label{corth:exval}
If a  two-person zero-sum game $\{\A,\bb, c \}$  introduced in Definition~\ref{defi:game} satisfies assumptions (i) and (ii) from Theorem~\ref{th:exval}, then
\begin{equation*}
\sup_{\pi^\bb\in\Delta(\bb)}{\hat{c}}^{\flat}(\pi^\bb)=\inf_{\pi^\A\in\mathbb{P}(\A)} {\hat{c}}^{\sharp}(\pi^\A)
\end{equation*}
for each $\Delta(\bb)\subset \mathbb{P}(\bb)$ such that $\mathbb{P}^{fs}(\bb)\subset\Delta(\bb)\subset \mathbb{P}^S(\bb).$
\end{corollary}
\begin{proof}  The corollary follows from Theorem~\ref{th:exval} and from $\mathbb{P}^{fs}(\bb)\subset\Delta(\bb)\subset \mathbb{P}^S(\bb).$
\qed\end{proof}

\begin{corollary}\label{cor:new1}
Let a  two-person zero-sum game $\{\A,\bb, c \}$  introduced in Definition~\ref{defi:game}
satisfy conditions~(i) and (ii) of Theorem~\ref{th:exval}. Then 
\begin{equation*}\label{eq:auauauavalval-lopsided:a1}
\inf_{\pi^\A\in\mathbb{P}(\A)} {\hat{c}}^{\sharp}(\pi^\A)=\min_{\pi^\A\in\mathbb{P}(\A)} {\hat{c}}^{\sharp}(\pi^\A).
\end{equation*} 
\end{corollary}
\begin{proof}
 The corollary follows from (\ref{eq:valval})
because the set $\mathbb{P}^\sharp_{v}(\A)$ in Theorem~\ref{th:exval} is
nonempty. \qed
\end{proof}

\begin{corollary}\label{cor:new2}
Let a  two-person zero-sum game $\{\A,\bb, c \}$  introduced in Definition~\ref{defi:game}
satisfy conditions~(i) and (ii) of Theorem~\ref{th:exval} and $c(\,\cdot\,,b)
$ be inf-compact  for each $b\in \bb.$   Then 
\begin{equation}\label{eq:new1}
\begin{aligned}
&\min_{\pi^\A\in\mathbb{P}(\A)}\sup_{\pi^\bb\in\mathbb{P}^{fs}(\bb)} {\hat c}(\pi^\A, \pi^\bb)=
&\sup_{\pi^\bb\in\mathbb{P}^{fs}(\bb)}\min_{\pi^\A\in\mathbb{P}(\A)} {\hat c}(\pi^\A, \pi^\bb).
\end{aligned}
\end{equation}
\end{corollary}
\begin{proof}
Observe that 
\[
\begin{aligned}
\min_{\pi^\A\in\mathbb{P}(\A)}&\sup_{\pi^\bb\in\mathbb{P}^{fs}(\bb)}{\hat c}(\pi^\A, \pi^\bb)=
 \inf_{\pi^\A\in\mathbb{P}(\A)}\sup_{\pi^\bb\in\mathbb{P}^{fs}(\bb)}{\hat c}(\pi^\A, \pi^\bb)\\ &=
 \sup_{\pi^\bb\in\mathbb{P}^{fs}(\bb)}\inf_{\pi^\A\in\mathbb{P}(\A)}{\hat c}(\pi^\A, \pi^\bb),
\end{aligned}
\]
where the first equality follows from Corollary~\ref{cor:new1}, Remark~\ref{rem:uncer}, and \eqref{eq:prop:mainIa}. The second equality follows from Corollary~\ref{corth:exval},
  applied to $\Delta(\bb)=\mathbb{P}^{fs}(\bb)$.
It remains to prove that, for each $\pi^\bb\in\mathbb{P}^{fs}(\bb),$
\begin{equation}\label{eq39}
\inf_{\pi^\A\in\mathbb{P}(\A)}{\hat c}(\pi^\A, \pi^\bb)=\min_{\pi^\A\in\mathbb{P}(\A)}{\hat c}(\pi^\A, \pi^\bb).
\end{equation}
 To prove \eqref{eq39} observe that the function $\pi^\A\mapsto{\hat c}(\pi^\A, b)$ on $\mathbb{P}(\A)$ is inf-compact for each $b\in\bb.$ 
This follows from Theorem~\ref{th:inf-comp}, applied to $\X:=\{b\},$ $\Y:=\A,$ and $f(b,a):=c(a,b),$ $a\in\A,$  because,  for each $b\in \bb,$ the function $c(\,\cdot\,,b)$ is inf-compact.
Equality  \eqref{eq39} follows from inf-compactness of the function $\pi^\A\mapsto{\hat c}(\pi^\A, \pi^\bb)$ on $\mathbb{P}(\A)$ for each $\pi^\bb\in\mathbb{P}^{fs}(\bb),$ which, in its turn,  follows from inf-compactness of the function $\pi^\A\mapsto{\hat c}(\pi^\A, b)$ on $\mathbb{P}(\A)$ for each $b\in\bb$ because the
convex combination of inf-compact functions ${\hat c}(\pi^\A, \pi^\bb)=\sum_{b\in B}  \pi^\bb(b){\hat c}(\pi^\A, b),$ where $B$ is a finite subset of $\bb$ such that $\pi^\bb(B)=1,$ is an inf-compact function.
\end{proof}

The following Corollary~\ref{th:mixed} to Theorem~\ref{th:exval} is Proposition~I.1.9 from Mertens et al. \cite{MSZ} for two-person zero-sum games $\{\A,\bb, c \}$  introduced in Definition~\ref{defi:game}. Note that the space $\A$ is a compact topological space, $\bb$ is any set, and for each $b\in \bb,$ $c(\,\cdot\,,b)$ is lower semi-continuous in Mertens et al. \cite[Proposition~I.1.9]{MSZ}.

\begin{corollary}\label{th:mixed}{\rm(cp.~Mertens et al. \cite[Proposition~I.1.9]{MSZ})}
Let $\{\A,\bb, c \}$ be a two-person zero-sum game  introduced in Definition~\ref{defi:game},
$\A$ be a compact, and for each $b\in \bb,$ $c(\,\cdot\,,b)$ is lower semi-continuous. Then
(\ref{eq:new1}) holds.
\end{corollary}
\begin{proof}
Since $\A$ be a compact, and for each $b\in \bb,$ $c(\,\cdot\,,b)$ is lower semi-continuous, then for each $b\in \bb,$ $c(\,\cdot\,,b)$ is inf-compact. Therefore, (\ref{eq:new1}) follows from Corollary~\ref{cor:new2}. \qed
\end{proof}



The following Proposition~\ref{auEknon} is Theorem~6.2.7 from Aubin and Ekeland~\cite{ObEk} for two-person zero-sum games $\{\A,\bb, c \}$  introduced in Definition~\ref{defi:game}. Note that the space $\A$ is a topological space, the space $\bb$ is not endowed with a topology, and $c$ is not measurable in Aubin and Ekeland~\cite[Theorem~6.2.7]{ObEk}.

\begin{proposition}{\rm(cp.~Aubin and Ekeland~\cite[Theorem~6.2.7]{ObEk})}\label{auEknon}
Let a  two-person zero-sum game $\{\A,\bb, c \}$  introduced in Definition~\ref{defi:game}
satisfy conditions~(i) and (ii) of Theorem~\ref{th:exval}. Suppose the spaces $\A$ and $\bb$ are convex subsets of vector spaces, the function $a\mapsto c(a,b)$ is convex for each $b\in\bb,$ and the function $b\mapsto c(a,b)$ is concave for each $a\in\A.$ 
Then
\begin{equation}\label{eq:new3}
\inf_{a\in \A}\sup_{b\in\bb}c(a,b)=\sup_{b\in\bb}\inf_{a\in\A} c(a,b)=:V.
\end{equation}
Moreover, there exists $a^*\in\A$ such that $\hat{c}^{\sharp}(a^*)=V.$
\end{proposition}

\begin{remark}
{\rm
 (i) The assumptions of Theorem~\ref{th:exval} are more general than the assumptions of Proposition~\ref{auEknon} because neither convexity nor concavity of the function $c$ is assumed in Theorem~\ref{th:exval}.
(ii) Under the assumptions of Proposition~\ref{auEknon}, 
the value $V$ equals the lopsided value defined in  \eqref{eq:valval}. 
This observation follows from the equality stated in Theorem~\ref{th:exval} and from the equalities
\begin{equation}\label{eq:conv1a}
\sup_{b\in\bb}\inf_{a\in\A} c(a,b)=\sup_{\pi^\bb\in\mathbb{P}^{fs}(\bb)}\hat{c}^{\flat}(\pi^\bb)
=\inf_{\pi^\A\in\mathbb{P}(\A)} \hat{c}^{\sharp}(\pi^\A)=\inf_{a\in \A}\sup_{b\in\bb}c(a,b),
\end{equation}
which follow from Proposition~\ref{auEknon} and 
\begin{equation}\label{eq:conv1a1a}
\sup_{b\in\bb}\inf_{a\in\A} c(a,b)\le\sup_{\pi^\bb\in\mathbb{P}^{fs}(\bb)}\hat{c}^{\flat}(\pi^\bb)
=\inf_{\pi^\A\in\mathbb{P}(\A)} \hat{c}^{\sharp}(\pi^\A)\le\inf_{a\in \A}\sup_{b\in\bb}c(a,b),
\end{equation}
where the inequalities in \eqref{eq:conv1a1a} hold because each actions $a\in \A$ and $b\in\bb$ for Players~I and II can be interpreted as the strategies $\delta_{\{a\}}\in\mathbb{P}(\A)$ and $\delta_{\{b\}}\in\mathbb{P}(\bb)$ concentrated in points $a$ and $b$ respectively, and  the equality in \eqref{eq:conv1a1a} follows from Theorem~\ref{th:exval}.
(iii) Under the assumptions of Theorem~\ref{th:exval}, there exists $a^*\in\A$ such that $\hat{c}^{\sharp}(a^*)=\inf_{a\in \A}\sup_{b\in\bb}c(a,b).$  This is true since  the function   $\hat{c}^{\sharp}(a)=\sup_{b\in\bb}c(a,b)$ is  inf-compact  on $\A$ because this function the supremum of lower semi-continuous functions $a\mapsto c(a,b)$ and at least one of them, $a\mapsto c(a,b_0),$ is  inf-compact.}
\end{remark}


\begin{remark}\label{rem:fr}
{\rm
Theorem~\ref{th:exval} allows the function $\hat c$ 
to take the values from $\mathbb{R}\cup\{+\infty\}$ unlike the payoff function in  Aubin and Ekeland~\cite[Theorem~6.2.7]{ObEk}, that takes only finite values. This is the reason why  Theorem~\ref{th:exval} does not follow from  \cite[Theorem~6.2.7]{ObEk}
and properties ($i_1$)--($i_4$) stated in the proof of Theorem~\ref{th:exval}.} 
\end{remark}
\begin{remark}
{\rm Fan's minimax theorem~\cite[Theorem 2]{Fan} states equality \eqref{eq:new3} for a convex-concave-like function $c,$ when $\A$ is a compact subset of a Hausdorff space, $\bb$ is arbitrary, and the functions $a\mapsto c(a,b)$ are lower semi-continuous for all $b\in\bb.$  By using this theorem, Perchet and Vigeral~\cite{SorSt1} provided \eqref{eq:new3} for a convex-concave function $c$ without the assumption that $\A$ is compact, but with additional assumptions including that $\A$ is finite-dimensional and bounded.
}
\end{remark}

The following example describes a two-person zero-sum game with noncompact action sets and unbounded payoffs  satisfying  the assumptions of Theorem~\ref{th:exval}. 

\begin{example}\label{exa:1}
{\rm Let $\A=\bb=\mathbb{R},$ $c(a,b)=a^2-b^2,$ $(a,b)\in \mathbb{R}^2.$ Then the game $\{\A,\bb, c \}$  satisfies the conditions of Theorem~\ref{th:exval} and $v=0.$
%
}
\end{example}

Example~\ref{exa:1} admits the following interpretation in the form of a simple game of timing (see Yanovskaya \cite[Section~6]{Yan}) with noncompact decision sets.  Two teams work on a project consisting of two independent tasks, each performed by one of the teams.  The project should be completed on a target date.  The project is completed when both tasks are completed, and they should be completed simultaneously.  The penalty, in the amount of $t^2$ paid to another team for completing its task by $t$ units of time later or earlier than the target date, creates incentives to the teams to complete their tasks exactly on time. Of course, there are other payoff functions including $|t|$ that provide incentives to achieve the same goal.

If
\begin{equation}\label{eq:valUvalS}
\sup_{\pi^\bb\in\mathbb{P}^U(\bb)}{\hat{c}}^{\flat}(\pi^\bb)\le \sup_{\pi^\bb\in\mathbb{P}^S(\bb)}{\hat{c}}^{\flat}(\pi^\bb),
\end{equation}
 as this takes place in Example~\ref{exa:1}, then the existence of the lopsided value $v$ defined in \eqref{eq:valval} implies that the equality
\begin{equation}\label{eq:valvalEFEF}
\sup_{\pi^\bb\in\mathbb{P}(\bb)}{\hat{c}}^{\flat}(\pi^\bb)=\inf_{\pi^\A\in\mathbb{P}(\A)} {\hat{c}}^{\sharp}(\pi^\A)
\end{equation}
holds.  In particular, \eqref{eq:valUvalS} and \eqref{eq:valvalEFEF} hold if ${\hat{c}}^{\flat}(\pi^\bb)=-\infty$ for all $\pi^\bb\in\mathbb{P}^U.$  The following example demonstrates that it is possible that under the condition, that  the function $(b,a)\mapsto c(a,b)$ is $\K$-inf-compact on $\bb\times\A,$ which is a stronger condition than the assumptions of Theorem~\ref{th:exval}, it is possible that ${\hat{c}}^{\flat}(\pi^\bb)>-\infty$ for some $\pi^\bb\in \mathbb{P}^U(\bb).$

%
%

\begin{example}\label{exa1}
{\rm The function $(b,a)\mapsto c(a,b)$ is $\K$-inf-compact on $\bb\times\A,$ the function $(a,b)\mapsto c(a,b)$ is $\K$-sup-compact on $\A\times\bb,$  and there exists $\pi^\bb\in \mathbb{P}^U(\bb)$ such that ${\hat{c}}^{\flat}(\pi^\bb)>-\infty.$

Let us set $\A:=\bb:=\{1,2,\ldots\},$ $c(a,b):=6^a4^b{\bf I}\{b<a\}-6^b4^a{\bf I}\{a<b\},$ $\pi^\bb(\{b\}):=\frac{11}{12^b},$ $b=1,2,\ldots\ .$ We consider the discrete metrics on $\A$ and $\bb.$ 

The function $(b,a)\mapsto c(a,b)$ is $\K$-inf-compact on $\bb\times\A$ because $c(a,b)\to+\infty,$ as $a\to\infty,$ for each $b=1,2,\ldots\ .$ Here we note that a set $K\subset \bb$ is compact if and only if   $K$ is finite. The function $(a,b)\mapsto c(a,b)$ is $\K$-sup-compact on $\A\times\bb$ because $c(a,b)\to-\infty,$ as $b\to\infty,$ for each $a=1,2,\ldots\ .$

We notice that for each $b=1,2,\ldots$
\[
\begin{aligned}
{\hat{c}}^\ominus(a,\pi^\bb)&=-\sum_{b=a+1}^\infty 6^b4^a\frac{11}{12^b}=-11\cdot 4^a\sum_{b=a+1}^\infty\frac{1}{2^b} 
=-11 \cdot2^a,\\
{\hat{c}}^\oplus(a,\pi^\bb)&=\sum_{b=1}^{a-1}6^a4^b\frac{11}{12^b}={11}\cdot 6^{a}\sum_{b=1}^{a-1} \frac{1}{3^b}=
\frac{11}{2}6^a-\frac{33}{2}2^a.
\end{aligned}
\]
 Therefore, ${\hat{c}}(a,\pi^\bb)=\frac{11}{2}6^a-\frac{55}{2} 2^a$ for each $a=1,2,\ldots\ .$ Since ${\hat{c}}(a,\pi^\bb)\to +\infty,$ as $a\to\infty,$ then ${\hat{c}}^{\flat}(\pi^\bb)=\inf_{a\in\A}{\hat{c}}(a,\pi^\bb)>-\infty.$

Let us set $\pi^\A(\{a\}):=\frac{1}{2^a},$ $a=1,2,\ldots\ .$
Since $\pi^\A\in\mathbb{P}(\A)$ and
\[
\begin{aligned}
{\hat{c}}^\ominus(\pi^\A,\pi^\bb)&=-\sum_{a=1,2,\ldots} 11 \cdot2^a\frac{1}{2^a}=-\infty,\\
{\hat{c}}^\oplus(\pi^\A,\pi^\bb)&=\sum_{a=1,2,\ldots}\left(\frac{11}{2}6^a-\frac{33}{2}2^a\right)\frac{1}{2^a}=+\infty,\\
\end{aligned}
\]
 then $\pi^\bb\in \mathbb{P}^U(\bb).$    
%
}
\end{example}

\subsection{The Existence of a Solution}\label{sub:sol}

This subsection provides the definition of a solution of a two-person zero-sum game with possibly non-compact actions and unbounded payoff. Theorem~\ref{teor:mainonestep} establishes sufficient conditions for the existence of solutions for such games.

\begin{definition}\label{defi:equil}
The pair of mixed strategies $(\pi^\A,\pi^\bb)\in \mathbb{P}^S(\A)\times\mathbb{P}^S(\bb)$ for Players~I and II is called a \textit{solution  (saddle point, equilibria)} of the game $\{\A,\bb, c\},$ if
\begin{equation}\label{eq:solution}
{\hat{c}}(\pi^\A,\pi^\bb_*)\le {\hat{c}}(\pi^\A,\pi^\bb)\le{\hat{c}}(\pi^\A_*,\pi^\bb)
\end{equation}
for each $\pi_*^\A\in \mathbb{P}(\A)$ and $\pi^\bb_*\in\mathbb{P}(\bb).$
%
\end{definition}

\begin{remark}\label{rem:sol}
{\rm Let the solution $(\pi^\A,\pi^\bb)\in \mathbb{P}^S(\A)\times\mathbb{P}^S(\bb)$ of the game $\{\A,\bb, c\}$ exist. Then the number
\begin{equation}\label{eq:rem15_1}
v:={\hat{c}}^{\flat}(\pi^\bb)={\hat{c}}^{\sharp}(\pi^\A)
\end{equation}
is the lopsided value of this game. Indeed, inequalities (\ref{eq:solution}) imply that 
 \begin{equation}\label{eq:rem15_2}
 {\hat{c}}^{\sharp}(\pi^\A)\le{\hat{c}}^{\flat}(\pi^\bb).
 \end{equation}

According to
 Remark~\ref{rem:main2} and Definition~\ref{Def5.8}, $(\pi^\A,\pi^\bb)\in \mathbb{P}^S(\A)\times\mathbb{P}^S(\bb)$ is a solution  of the game $\{\A,\bb, c\}$ 
 if and only if inequality (\ref{eq:rem15_2}) holds. Indeed, if $(\pi^\A,\pi^\bb)\in \mathbb{P}^S(\A)\times\mathbb{P}^S(\bb)$
 is the solution of the game $\{\A,\bb, c\},$ then inequalities (\ref{eq:solution}) imply (\ref{eq:rem15_2}). Vice versa, if  inequality (\ref{eq:rem15_2}) holds, then, since
 $\mathbb{P}^S_{\pi^\A}(\bb)=\mathbb{P}(\bb)$ and $\mathbb{P}^S_{\pi^\bb}(\A)=\mathbb{P}(\A)$,   Theorem~\ref{cor:mainI}  implies inequalities (\ref{eq:solution}), that is, $(\pi^\A,\pi^\bb)\in \mathbb{P}^S(\A)\times\mathbb{P}^S(\bb)$
 is the solution of the game $\{\A,\bb, c\}.$ We remark also that inequality (\ref{eq:rem15_2}) holds if and only if $\pi^\A\in \mathbb{P}^\sharp_{v}(\A)$ and  $\pi^\bb\in \mathbb{P}_{v}^\flat(\bb)$ because  of (\ref{eq:rem15_1}) and the definitions of $\mathbb{P}^\sharp_{v}(\A)$ and $\mathbb{P}_{v}^\flat(\bb).$ Furthermore, according to Remark~\ref{rem:lopsided}, in the case of $\mathbb{P}^S(\bb)=\mathbb{P}(\bb),$ which takes place in Theorems~\ref{teor:mainonestep}, \ref{teor:maincont of equil} and Corollary~\ref{corteor:maincont of equil}, the lopsided value is equal to the value in the classic sense.}
\end{remark}

The following theorem provides sufficient conditions for the existence of a solution.

\begin{theorem}\label{teor:mainonestep}
Let  a two-person zero-sum game $\{\A,\bb, c \}$ introduced in Definition~\ref{defi:game} satisfy the following assumptions:
\begin{itemize}
\item[{\rm(a)}] the function $(b,a)\mapsto c(a,b)$ is $\K$-inf-compact on $\bb\times\A;$
\item[{\rm(b)}] the function $(a,b)\mapsto c(a,b)$ is $\K$-sup-compact on $\A\times\bb;$
\item[{\rm(c)}] the function $(a,b)\mapsto c(a,b)$ is bounded from below.
\end{itemize}

 Then the following statements hold:
\begin{itemize}
\item[{\rm(i)}] the game $\{\A,\bb, c \}$ has a solution $(\pi^\A,\pi^\bb)\in \mathbb{P}^\sharp_{v}(\A)\times \mathbb{P}_{v}^\flat(\bb);$
\item[{\rm(ii)}]  the sets  $\mathbb{P}^\sharp_{v}(\A)$ and $\mathbb{P}_{v}^\flat(\bb)$
 are nonempty convex compact subsets of $\mathbb{P}(\A)$ and $\mathbb{P}(\bb)$ respectively;  
\item[{\rm(iii)}] a pair of
  strategies $(\pi^\A,\pi^\bb)\in \mathbb{P}(\A)\times \mathbb{P}(\bb)$
  is a solution of the game $\{\A,\bb, c \}$ if and only if
  $\pi^\A\in \mathbb{P}^\sharp_{v}(\A)$ and  $\pi^\bb\in \mathbb{P}_{v}^\flat(\bb).$
  \end{itemize}
\end{theorem}
\begin{proof}
Assumptions (b) and Theorem~\ref{th:exval} imply that the game $\{\A,\bb, c \}$ has the lopsided value and $\mathbb{P}^\sharp_{v}(\A)$ is a
  nonempty convex compact subset of $\mathbb{P}(\A).$  In view of Remark~\ref{rem:equivEFAB}, assumption (c) implies that $\mathbb{P}^S(\bb)=\mathbb{P}(\bb)$ and equality \eqref{rem:lopsidedclas} 
   holds. In view of Remark~\ref{rem:lopsided}, this game has the value. Assumption (b) and
Theorem~\ref{th:exval}, being applied to the game $\{\bb,\A, -c^{\A\leftrightarrow\bb} \},$ where $c^{\A\leftrightarrow\bb}(b,a):=c(a,b)$ for each $a\in\A$ and $b\in\bb,$ imply that  the set $\mathbb{P}_{v}^\flat(\bb)$
 is a nonempty convex compact subset of  $\mathbb{P}(\bb).$  Thus, statement (ii) is proved.  Statements (i) and (iii) follow from Remark~\ref{rem:sol}.
\qed \end{proof}

\begin{remark}\label{rem:compg}
{\rm
Assumptions  (b) and (c) of Theorem~\ref{teor:mainonestep} imply that the space of actions $\bb$ for Player~II is compact.
}
\end{remark}
\begin{remark}\label{rem:compgEF}
{\rm
As the proof of Theorem~\ref{teor:mainonestep} shows, assumptions (a) and (b) of Theorem~\ref{teor:mainonestep} can be relaxed.  Assumption (a) can be relaxed to the pair of assumptions (i, ii) from Theorem~\ref{th:exval}.  Assumption (b) can be relaxed to the pair of assumptions symmetric to assumptions (i) and (ii) from Theorem~\ref{th:exval}: for each $a\in \A$ the function $b\mapsto c(a,b)$ is upper semi-continuous, and
there exists $a_0\in \A$ such that the function $b\mapsto c(a_0,b)$ is inf-compact on $\bb.$
}
\end{remark}

\subsection{Continuity Properties of Equilibria}\label{s5.4}

In this section we define and study families of games with action sets and payoff functions depending on a parameter. Let $\X,$ $\A$ and $\bb$ be Borel subsets of Polish spaces, $K_\A\in {\mathcal B}(\X\times\A),$
where ${\mathcal B}(\X\times\A)={\mathcal B} (\X)\otimes {\mathcal
B}(\mathbb{A}),$ $K_\bb\in {\mathcal B}(\X\times\bb),$
where ${\mathcal B}(\X\times\bb)={\mathcal B} (\X)\otimes {\mathcal
B}(\bb).$ It is assumed that for each $x\in\X$ the sets $K_\A$ and $K_\bb$ satisfy the following two conditions:
\[
 A(x):=\{a\in\A\,:\, (x,a)\in K_\A\}\ne\emptyset\quad{\rm and}\quad
 B(x):=\{b\in\bb\,:\, (x,b)\in K_\bb\}\ne\emptyset.
\]

Let \[
\kk:=\{(x,a,b)\in \X\times\A\times\bb\,:\, x\in \X,\, a\in A(x),\, b\in B(x)\}.
\]

\begin{remark}\label{rem:kk}
{\rm We note that $\Gr(A)=K_\A,$ $\Gr(B)=K_\bb,$ and $\kk=\Gr(A\times B),$ where $(A\times B)(x) :=\{(a,b)\,:\, a\in A(x), \, b\in B(x)\},$ $x\in \X.$ We note also that $\kk=\Gr(\tilde{B}),$ where $\tilde{B} (x,a) :=B(x),$ $(x,a)\in K_\A.$ If we set $\tilde{A} (x,b): =A(x),$ $(x,b)\in K_\bb,$ then $\Gr(\tilde{A})=\{(x,b,a)\,:\,(x,a,b)\in\kk\}$ and $\kk=\{(x,a,b)\,:\, (x,b,a)\in\Gr(\tilde{A})\}.$}
\end{remark}

Consider the family of   two-person zero-sum games
\[
\{\{A(x),B(x), c(x,\,\cdot\,,\,\cdot\,)\}\,:\,x\in\X \} 
\]
satisfying for each $x\in\X$ all the assumptions from Definition~\ref{defi:game}.
Define the function $c^{\A\leftrightarrow \bb}:\Gr(\tilde{A})\subset (\X\times\bb)\times\A\mapsto\R,$ \begin{equation}\label{eq:auxil3silm}
c^{\A\leftrightarrow \bb}(x,b,a):=c(x,a,b),\quad (x,a,b)\in \mathcal{K}.
\end{equation}

Let us consider the following assumptions. 

\noindent {\bf Assumption (A1)} The function $c^{\A\leftrightarrow \bb}:\Gr(\tilde{A})\subset (\X\times\bb)\times\A\mapsto\mathbb{R}$ defined in (\ref{eq:auxil3silm}) is $\K$-inf-compact on $\Gr(\tilde{A}).$ 

\noindent {\bf Assumption (A2)} The function $c:\kk\subset (\X\times\A)\times\bb\mapsto\mathbb{R}$  is $\K$-sup-compact on $\kk.$  

\noindent {\bf Assumption (A3)} $A:\X\mapsto S(\A)$ is a lower semi-continuous set-valued mapping.

\noindent {\bf Assumption (A4)} $B:\X\mapsto S(\bb)$ is a lower semi-continuous set-valued mapping.

\begin{remark}\label{rem:kinfcomp}
{\rm According to Lemma~\ref{k-inf-compact} and Remarks~\ref{rem:substit}, \ref{rem:kk}, Assumption (A1) holds if and only if the following two conditions hold:  
\begin{itemize}
\item[(i)] the mapping $c:\kk\subset\X\times\A\times\bb\mapsto\mathbb{R}$ is lower semi-continuous; 
\item[(ii)] if a sequence $\{x^{(n)},b^{(n)}\}_{n=1,2,\ldots}$ with values in $K_\bb$
converges and its limit $(x,b)$ belongs to $K_\bb,$ then each sequence $\{a^{(n)}
\}_{n=1,2,\ldots}$ with $(x^{(n)},a^{(n)},b^{(n)})\in \kk,$ $n=1,2,\ldots,$ satisfying
the condition that the sequence\\ $\{c(x^{(n)},a^{(n)},b^{(n)}) \}_{n=1,2,\ldots}$ is
bounded above, has a limit point $a\in A(x).$
\end{itemize}}
\end{remark}
\begin{remark}\label{rem:ksupcomp}
{\rm According to Lemma~\ref{k-inf-compact} and Remark~\ref{rem:kk}, Assumption (A2) holds if and only if the following two conditions hold:
\begin{itemize}
\item[(i)] the mapping $c:\kk\subset\X\times\A\times\bb\mapsto\mathbb{R}$ is upper semi-continuous; 
\item[(ii)] if a sequence $\{x^{(n)},a^{(n)} \}_{n=1,2,\ldots}$ with values in $K_\A$
converges and its limit $(x,a)$ belongs to $K_\A,$ then each sequence $\{b^{(n)}
\}_{n=1,2,\ldots}$ with\\ $(x^{(n)},a^{(n)},b^{(n)})\in \kk,$ $n=1,2,\ldots,$ satisfying
the condition that the sequence $\{c(x^{(n)},a^{(n)},b^{(n)}) \}_{n=1,2,\ldots}$ is
bounded from below, has a limit point $b\in B(x).$
\end{itemize}}
\end{remark}
\begin{remark}\label{remFINc} 
{\rm   Assumptions~(A1) and (A2) imply that the {payoff} to  Player II, 
$ c(x,a,b)$ for choosing actions $a\in A(x)$ and $b\in B(x)$ in a state $x\in\X,$ is  continuous.} 
\end{remark}

\begin{remark}\label{rembarAEF} {\rm
If the function $c$ takes values in $\overline{\mathbb{R}}$ instead of  $\mathbb{R}$ in Assumptions~(A1) and (A2),  then Remarks~\ref{rem:kinfcomp} and \ref{rem:ksupcomp} are also applicable to such functions.  However, we consider only real-valued payoff functions $c$ in this paper.} 
\end{remark} 

Let $\{\{A(x),B(x), c(x,\,\cdot\,,\,\cdot\,)\}\,:\,x\in\X \}$ be  the family of   two-person zero-sum games, that is, each of these games satisfies assumptions in Definition~\ref{defi:game}. Further let 
${\hat{c}}^{\sharp}(x)$ and ${\hat{c}}^{\flat}(x)$ be defined in (\ref{eq:maxmin}) and $v(x)$ denote the  lopsided value of the game $\{A(x),B(x), c(x,\cdot,\cdot)\}$ if it exists, 
$x\in\X$ (in Theorem~\ref{teor:maincont of equil} $v(x)$ is the value). 

The following theorem provides sufficient conditions for the lower semi-continuity of the lopsided value for a family of two-person zero-sum games with
possibly noncompact action sets and unbounded payoffs. 
\begin{theorem}\label{teor:mainlsc of equil}
Let the family of   two-person zero-sum games\\ $\{\{A(x),B(x), c(x,\,\cdot\,,\,\cdot\,)\}\,:\,x\in\X \}$ satisfy Assumptions~{\rm(A1)} and {\rm(A4)}.  Then the following statements hold:
\begin{itemize}
\item[{\rm(i)}] for each $x\in\X$ the following equality holds: 
\begin{equation}\label{eq:v1}
\sup\limits_{\pi^\bb\in \mathbb{P}^S(B(x))}{\hat{c}}^{\flat}(x,\pi^\bb)=\inf\limits_{\pi^\A\in \mathbb{P}(A(x))} {\hat{c}}^{\sharp}(x,\pi^\A) \,\left(=:v(x)\right).
\end{equation}
Moreover, $v:\X\mapsto\mathbb{R}$ is a lower semi-continuous function;
\item[{\rm(ii)}] the sets $\{\mathbb{P}^\sharp_{v(x)}(A(x))\,:x\in\X\}$
satisfy the following properties:
\begin{itemize}
\item[{\rm(a)}]   
for each $x\in\X$ the set
$\mathbb{P}^\sharp_{v(x)}(A(x))$  is a nonempty convex compact subset of $\mathbb{P}(\A);$
\item[{\rm(b)}]
the graph ${\rm Gr}(\mathbb{P}^\sharp_{v(\,\cdot\,)}(A(\,\cdot\,)))=\{(x,\pi^\A):\, x\in\X, \pi^\A\in
\mathbb{P}^\sharp_{v(x)}(A(x))\}$ is a Borel subset of   $\X\times \mathbb{P}(\A);$
\item[{\rm(c)}] there exists a measurable mapping $\phi^\A:\X\mapsto\mathbb{P}(\A)$ such that $\phi^\A(x)\in \mathbb{P}^\sharp_{v(x)}(A(x))$ for each $x\in\X.$
\end{itemize}
\end{itemize}
\end{theorem}

\begin{proof}
Assumption (A1) and Corollary~\ref{cor:K-inf-comp}, being applied to $\X:=\X\times\bb$ (that is, the state space is $\X\times\bb$),  $\Y:=\A,$ $f:=c^{\A\leftrightarrow \bb}$ on $\Gr(\tilde{A}),$ and $f:=+\infty$ on the complement of $\Gr(\tilde{A}),$ imply that the mapping $\hat{c}^{\A\leftrightarrow \bb}:\Gr(\mathbb{P}(\tilde{A}(\,\cdot\,,\,\cdot\,)))\subset (\X\times\bb)\times\mathbb{P}(\A)\mapsto\mathbb{R},$ where  
\[
\hat{c}^{\A\leftrightarrow \bb}(x,b,\pi^\A)= \int_{A(x)} c(x,a,b)\pi^\A(da),\  (x,b)\in \K_\bb,\, \pi^\A\in \mathbb{P}(\tilde{A}(x,b))=\mathbb{P}({A}(x)), 
\]
is $\K$-inf-compact on $\Gr(\mathbb{P}(\tilde{A}(\,\cdot\,,\,\cdot\,))).$ Identity~(\ref{eq:v1}) follows from Theorem~\ref{th:exval}. The  remaining statements 
follow from Theorem~\ref{th:minimax_lsc}, being applied to $\Xx:=\X,$ $\Aa:=\mathbb{P}(\A),$ $\Bb:=\bb,$ $\Phii_\Aa(\,\cdot\,):=\mathbb{P}(A(\,\cdot\,)),$  $\Phii_\Bb(x,\pi^\A):=B(x),$ $x\in\X,$ and $f(x,\pi^\A,b):={\hat{c}}(x,\pi^\A,b),$ $(x,\pi^\A,b)\in\{(x,\pi^\A,b)\in \X\times\mathbb{P}(\A)\times\bb\,:\, (x,b)\in\K_\bb,\, \pi^\A\in\mathbb{P}({A}(x))\},$ from Lemma~\ref{lem:unifAlsc}, and from Feinberg et al. \cite[Theorem~3.3]{Feinberg et al}.
\qed \end{proof} 

The following example describes a family of two-person zero-sum games satisfying Assumptions (A1) and (A4).  Payoff functions are unbounded and decision sets are noncompact for the games in this family.
\begin{example}\label{exa:noncomp}
{\rm
Let $\X=\A=\bb=\mathbb{R},$ $K_\A=K_\bb=\mathbb{R}^2,$ $\mathcal{K}=\mathbb{R}^3,$ $c(x,a,b)=\varphi_\X(x)+\varphi_\A(a)+\varphi_\bb(b),$ $(x,a,b)\in \mathcal{K},$ where
$\varphi_\X,\varphi_\A,\varphi_\bb:\mathbb{R}\mapsto \mathbb{R}$ are continuous functions such that
$\varphi_\A(a)\to+\infty$ as $|a|\to\infty.$ Then $c$ is a continuous function on $\mathbb{R}^3$ and it satisfies Assumption (A1). Indeed, let a sequence $\{x^{(n)},b^{(n)} \}_{n=1,2,\ldots}$ with values in $\mathbb{R}^2$
converges and its limit $(x,b)$ belongs to $\mathbb{R}^2,$ a sequence $\{a^{(n)}
\}_{n=1,2,\ldots}$ with $(x^{(n)},a^{(n)},b^{(n)})\in\mathbb{R}^3,$ $n=1,2,\ldots,$ satisfy
the condition that the sequence $\{c(x^{(n)},a^{(n)},b^{(n)}) \}_{n=1,2,\ldots}$ is
bounded above. Then the sequence $\{\varphi_\A(a^{(n)}) \}_{n=1,2,\ldots}$
is
bounded above and, since $\varphi_\A(a)\to+\infty$ as $|a|\to\infty,$ then the sequence $\{a^{(n)}
\}_{n=1,2,\ldots}$
has a limit point $a\in A(x)=\mathbb{R}.$ Therefore, Assumption (A1) holds. Assumption (A4) holds, because the multi-valued mapping $\Phi:\mathbb{R}\mapsto S({\mathbb{R}}),$ $\Phi(s)=\mathbb{R},$ $s\in\mathbb{R},$ is lower semi-continuous on $\mathbb{R}.$
}
\end{example}

The following theorem and its corollary describes sufficient conditions for continuity of the value function and upper semi-continuity of the solution multifunctions for a family of   two-person zero-sum games with possibly noncompact action sets and unbounded payoffs.

\begin{theorem}\label{teor:maincont of equil}{\rm(Continuity of equilibria)}
Let a family of   two-person zero-sum games $\{\{A(x),B(x), c(x,\,\cdot\,,\,\cdot\,)\}\,:\,x\in\X \}$ satisfy Assumptions~{\rm(A1)--(A4)} and $\bb$ be  compact. Then the following statements hold: 
\begin{itemize}
\item[{\rm(i)}] for each $x\in\X$
the game $\{A(x),B(x), c(x,\,\cdot\,,\,\cdot\,)\}$ has a solution
$(\pi^\A,\pi^\bb)\in \mathbb{P}^\sharp_{v(x)}(A(x))\times \mathbb{P}_{v(x)}^\flat(B(x)).$ Moreover, 
$v:\X\mapsto\mathbb{R}$ is a continuous function;
%

\item[{\rm(ii)}] the sets $\{\mathbb{P}^\sharp_{v(x)}(A(x)):x\in\X\}$
satisfy the following properties:
\begin{itemize}
\item[{\rm(a)}] for each $x\in \X$ the set
$\mathbb{P}^\sharp_{v(x)}(A(x))$  is a nonempty convex compact subset of $\mathbb{P}(\A);$ 
\item[{\rm(b)}] the multifunction $\mathbb{P}^\sharp_{v(\,\cdot\,)}(A(\,\cdot\,)):\X\mapsto \K(\mathbb{P}(\A))$ is upper semi-continuous;
\end{itemize}
\item[{\rm(iii)}] the sets $\{\mathbb{P}^\flat_{v}(B(x)):x\in\X\}$
satisfy the following properties:
\begin{itemize}
\item[{\rm(a)}] for each $x\in \X$ the set
$\mathbb{P}^\flat_{v(x)}(B(x))$  is a nonempty convex compact subset of $\mathbb{P}(\bb);$ 
\item[{\rm(b)}] the multifunction $\mathbb{P}^\flat_{v(\,\cdot\,)}(B(\,\cdot\,)):\X\mapsto \K(\mathbb{P}(\bb))$ is upper semi-continuous.
    \end{itemize}
\end{itemize}
\end{theorem}

\begin{proof}
In view of Theorem~\ref{teor:mainonestep} and Remark~\ref{rem:compg}, Theorem~\ref{teor:mainlsc of equil}, being applied to $\{\{A(x),B(x), c(x,\,\cdot\,,\,\cdot\,)\}\,:\,x\in\X \}$ and $\{\{B(x),A(x), -c^{\A\leftrightarrow\bb}(x,\,\cdot\, ,\,\cdot\,)\}\,:\,x\in\X \},$ where $c^{\A\leftrightarrow\bb}(x,b,a):=c(x,a,b)$ for each $x\in\X,$ $a\in A(x)$ and $b\in B(x),$ implies all the statements of the theorem.
\qed \end{proof} 

\begin{corollary}\label{corteor:maincont of equil}
Let a family of   two-person zero-sum games\\ $\{\{A(x),B(x), c(x,\,\cdot\,,\,\cdot\,)\}\,:\,x\in\X \}$  satisfy assumptions of Theorem~\ref{teor:maincont of equil}. Then there exist measurable mappings $\phi^\A:\X\mapsto\mathbb{P}(\A)$ and $\phi^\bb:\X\mapsto\mathbb{P}(\bb)$
such that $\phi^\A(x)\in \mathbb{P}^\sharp_{v(x)}(A(x))$ and $\phi^\bb(x)\in \mathbb{P}^\flat_{v(x)}(B(x))$ for all $x\in\X.$ Moreover, for each $x\in\X$
     a pair of 
  strategies $(\pi^\A(x),\pi^\bb(x))\in \mathbb{P}(A(x))\times \mathbb{P}(B(x))$
  is a solution of the game $\{A(x),B(x), c(x,\cdot,\cdot) \}$ if and only if
  $\pi^\A(x)\in \mathbb{P}^\sharp_{v(x)}(A(x))$ and  $\pi^\bb(x)\in \mathbb{P}_{v(x)}^\flat(B(x)).$
  \end{corollary}
\begin{proof}
All statements directly follow from  statements (ii) and (iii) of Theorem~\ref{teor:maincont of equil}.
\qed \end{proof}

\section{Notes on One-Step Two-Person Zero-Sum Stochastic Games with Perfect Information}\label{sec:perf}

This section shows that for the sequential one-step game studied in Section~\ref{sec:minimax}, it is sufficient for the both players to use only pure strategies.

Let $\X,$ $\A,$ and $\bb$ be Borel subsets of Polish spaces, $\Phi_{\A}:\X\mapsto S(\A)$ and $\Phi_{\bb}:\Gr(\Phi_{\A})\subset\X\times\A\mapsto S(\bb)$ be set-valued mappings and $f:{\rm Gr}(\Phi_{\bb})\subset \X \times \A\times\bb \mapsto \overline{\mathbb{R}}$ be a function.
  A \textit{one-step two-person zero-sum stochastic game with perfect information} is a tuple $\{\X, \A,\bb, \Phi_{\A},\Phi_{\bb}, f\}$ satisfying the following assumptions:
\begin{itemize}
\item[{(i)}]  $\X$ is the state space;
\item[{(ii)}] $\A$ is the \textit{action space} of the \textit{Player I};
\item[{(iii)}] $\bb$ is the \textit{action space} of the \textit{Player~II};
\item[{(iv)}] $\Gr(\Phi_\A)\in {\mathcal B}(\X\times\A),$
where ${\mathcal B}(\X\times\A)={\mathcal B} (\X)\otimes {\mathcal
B}(\mathbb{A}),$ is the \textit{constrained set} for the \textit{Player I}. It is assumed the existence of a measurable mapping $\phi_\A:\X\mapsto \mathbb{A}$ such
that $\phi_\A(x)\in \Phi_\A(x)$ for each $x\in\X.$ A nonempty Borel subset $\Phi_\A(x)$ of $\mathbb{A}$ represents the
\textit{set of admissible actions} of the \textit{Player~I} in the state $x\in \X;$
\item[{(v)}] $\Gr(\Phi_\bb)\in {\mathcal B}(\X\times\A\times\bb),$
where ${\mathcal B}(\X\times\A\times\bb)={\mathcal B} (\X)\otimes{\mathcal B} (\A)\otimes {\mathcal
B}(\bb),$ is the \textit{constrained set} for the \textit{Player II}. It is assumed the existence of a measurable mapping $\phi_\bb:\X\times\A\mapsto \mathbb{B}$ such
that $\phi_\bb(x,a)\in \Phi_\bb(x,a)$ for each $(x,a)\in \Gr(\Phi_\A).$ A nonempty Borel subset $\Phi_\bb(x,a)$ of $\mathbb{B}$ represents the
\textit{set of admissible actions} of the \textit{Player~II} in the state $x\in \X$ when Player~I choose an action $a\in\Phi_\A(x);$
\item[{(vi)}] the \textit{stage cost} for Player I, $-\infty\le f(x,a,b)\le +\infty,$ for choosing actions $a\in \Phi_\A(x)$ and $b\in \Phi_\bb(x,a)$ in a state $x\in\X,$ is a \textit{Borel} function on $\Gr(\Phi_\bb).$
\end{itemize}

\textit{The decision process proceeds as follows}:

$\bullet$  the current state $x\in\X$
is observed by each player;

$\bullet$ Player I choose an action $a\in \Phi_\A(x);$

$\bullet$ the result $a$ is announced to Player II;

$\bullet$ Player II choose an action $b\in  \Phi_\bb(x,a);$

$\bullet$ the result $b$ is announced to Player I;

$\bullet$ Player I pays Player II the amount $f(x,a,b).$

For a one-step two-person zero-sum stochastic game with perfect information $\{\X, \A,\bb, \Phi_{\A},\Phi_{\bb}, f\},$ let $f^\sharp$ be the worst-loss function (for Player~I) defined in (\ref{eq1starworstloss}), $v^\sharp$
be the minimax function defined in (\ref{eq1starminimax}),
and $\Phi_{\A}^*$ and $\Phi_{\bb}^*$ be the solution multifunctions defined in (\ref{e:defFi*minimax1}) and (\ref{e:defFi*minimax2})
respectively. If for each $(x,a)\in \Gr(\Phi_{\A})$ the function $b\mapsto f(x,a,b)$ is bounded from above, then, according to Theorem~\ref{cor:mainI}, the following equalities hold:
\begin{equation}\label{eq:perf1}
\sup_{\pi^\bb\in\mathbb{P}(\Phi_{\bb}(x,a))} \int_{\Phi_{\bb}(x,a)}f(x,a,b)\pi^\bb(db)=\sup_{b\in \Phi_{\bb}(x,a)} f(x,a,b)=f^\sharp(x,a),
\end{equation}
for each $(x,a)\in \Gr(\Phi_{\A}).$ Moreover, if for each $x\in\X$ the function $a\mapsto f^\sharp(x,a)$ is bounded from below, then, according to Theorem~\ref{cor:mainI}, the following equalities additionally hold:
\begin{equation}\label{eq:perf2}
\inf_{\pi^\A\in \mathbb{P}(\Phi_\A(x))}\int_{\Phi_\A(x)}f^\sharp(x,a)\pi^\A(da)=\inf\limits_{a\in \Phi_{\A}(x)}f^\sharp(x,a)=v^\sharp(x),
\end{equation}
for each $x\in\X.$ Therefore, all theorems and corollary from Section~\ref{sec:minimax} hold for stochastic one-step two-person zero-sum stochastic game with perfect information $\{\X, \A,\bb, \Phi_{\A},\Phi_{\bb}, f\}$ when each player possibly choose mixed strategies. According to equalities (\ref{eq:perf1}) and (\ref{eq:perf2}), the optimas for each player attain on the sets of respective pure strategies.

\section*{Appendix\quad Properties of $\Aa$-Lower Semi-Continuous Multifunctions}

\renewcommand{\appendix}{\setcounter{section}{0}}\renewcommand{\thesection}{A}

\appendix

This appendix describes some properties of $\Aa$-lower semi-continuous multifunctions.    Definition~\ref{defi:uniformAlsc} and the definition of lower semi-continuous multifunctions imply that an $\Aa$-lower semi-continuous multifunction is lower semi-continuous.  The following example demonstrates that a lower semi-continuous multifunction may not be $\Aa$-lower semi-continuous.

\begin{example}\label{exa:Ap}
{\rm Let $\Xx=\Aa=[0,1],$ $\Bb=\mathbb{R},$ $\Phii_\Aa(x)=\{x\}\cup\{\frac{1}{x}\}$ for $x\in (0,1],$ $\Phii_\Aa(0)=\{0\},$ and $\Phii_\Bb(x,a)=\{a\}$ for all $(x,a)\in {\rm Gr}(\Phii_\Aa).$  Since each set $\Phii_\Bb(x,a)$ is a singleton, where $(x,a)\in {\rm Gr}(\Phii_\Aa),$ and the graph of the multifunction $\Phii_\Bb$ is closed,  the multifunction $\Phii_\Bb$ is lower semi-continuous.  Let us consider the sequence $\{x_n\}_{n=1,2,\ldots}=\{\frac{1}{n}\}_{n=1,2,\ldots}$  converging to $x=0.$  Then $b:=0\in \Phii_\Bb(0,0)$ and $a^{(n)}=n\in \Phii_\Aa(x^{(n)}),$ $n=1,2,\ldots\ .$  However, the sequence $\{b_n\}_{n=1,2,\ldots}:=\{n\}_{n=1,2,\ldots}$ does not have a limit point.  Thus, the multifunction $\Phii_\Bb$ is not $\Aa$-lower semi-continuous.
}
\end{example}

Let us provide sufficient conditions for   $\Aa$-lower semi-continuity.

\begin{lemma}\label{lem:unifAlsc}
Let $\Phii_{\Bb}:\Gr(\Phii_{\Aa})\subset\Xx\times\Aa\mapsto S(\Bb)$ be a lower semi-continuous set-valued mapping. Then the following statements hold:
\begin{itemize}
\item[{\rm(a)}] if $\Phii_{\Aa}:\Xx\mapsto S(\Aa)$ is upper semi-continuous and compact-valued at each $x\in\Xx,$ then $\Phii_{\Bb}:\Gr(\Phii_{\Aa})\subset\Xx\times\Aa\mapsto S(\Bb)$ is  $\Aa$-lower semi-continuous;
\item[{\rm(b)}] if $\Phii_{\Bb}(x,a)$ does not depend on $a\in \Phii_\Aa(x)$ for each $x\in\Xx,$ that is, $\Phii_{\Bb}(x,a_*)=\Phii_{\Bb}(x,a^*)$ for each $(x,a_*),(x,a^*)\in \Gr(\Phii_\Aa),$ then $\Phii_{\Bb}:\Gr(\Phii_{\Aa})\subset\Xx\times\Aa\mapsto S(\Bb)$ is  $\Aa$-lower semi-continuous.
\end{itemize}
\end{lemma}
\begin{remark}\label{rem:uA-lsc}
{\rm Let $\Phii:\Xx\mapsto S(\Bb),$ where $\Phii(x)$ can be interpreted as the set of actions for Player II, when this set does not depend on the actions of Player I, as this takes place for games with simultaneous moves.  Then we can define the sets
\begin{equation}\label{eqdefphphi}
\Phii_\Bb(x,a):= \Phii(x),\qquad\qquad(x,a)\in{\rm Gr}(\Phii_\Aa).
\end{equation}
The definition of a lower semi-continuous multifunction implies that, if the multifunction $\Phii:\Xx\mapsto S(\Bb)$ is lower semi-continuous, then the multifunction $\Phii_{\Bb}:\Gr(\Phii_{\Aa})\subset\Xx\times\Aa\mapsto S(\Bb)$ is lower semi-continuous too.
Lemma~\ref{lem:unifAlsc} implies that  the lower semi-continuity of $\Phii_{\Bb}:\Gr(\Phii_{\Aa})\subset\Xx\times\Aa\mapsto S(\Bb)$ is equivalent to its   $\Aa$-lower semi-continuity in the following two cases:
(a) for two-person zero-sum games with perfect information, when the decision sets $\{\Phii_{\Aa}(x)\}_{x\in\Xx}$ for the first player are compact and the dependence of $\Phii_{\Aa}(x)$ by the state variable $x$ is upper semi-continuous, and (b) for two-person zero-sum games with simultaneous moves.
}
\end{remark}
\begin{proof}{\it of Lemma~\ref{lem:unifAlsc}}
(a) Let $\{x^{(n)}\}_{n=1,2,\ldots}$ be a sequence with values in $\Xx$ that converges and its limit $x$ belongs to $\Xx.$ Let also $a^{(n)}\in \Phii_\Aa(x^{(n)}),$ for each $n=1,2,\ldots,$ and $b\in \Phii_\Bb(x,a)$ for some $a\in \Phii_\Aa(x).$ Let us prove that
$b$ is a limit point for a sequence $\{b^{(n)}\}_{n=1,2,\ldots}$ with $b^{(n)}\in \Phii_\Bb(x^{(n)},a^{(n)})$ for each $n=1,2,\ldots\ .$ Indeed, Lemma~\ref{k-upper semi-comp}, being applied to $\X:=\Xx,$ $\Y:=\Aa,$ and $\Phi:=\Phii_\Aa,$ implies that the sequence $\{a^{(n)}\}_{n=1,2,\ldots}$ has a limit point $a\in \Phii_\Aa(x).$ Therefore, $b$ is a limit point of a sequence $\{b^{(n)}\}_{n=1,2,\ldots}$ with $b^{(n)}\in \Phii_\Bb(x^{(n)},a^{(n)})$ for each $n=1,2,\ldots,$ since $\Phii_{\Bb}:\Gr(\Phii_{\Aa})\subset\Xx\times\Aa\mapsto S(\Bb)$ is a lower semi-continuous set-valued mapping.

(b) Since $\Phii_{\Bb}:\Gr(\Phii_{\Aa})\subset\Xx\times\Aa\mapsto S(\Bb)$ is a lower semi-continuous set-valued mapping and $\Phii(x,a)$ does not depend on $a\in \Phii_\Aa(x)$ for each $x\in\Xx,$  the following statement holds: if a sequence $\{x^{(n)}\}_{n=1,2,\ldots}$ with values in $\Xx$ converges and its limit $x$ belongs to $\Xx,$ $a^{(n)}\in \Phii_\Aa(x^{(n)})$ for each $n=1,2,\ldots,$ and $b\in \Phii_\Bb(x,a)$ for some $a\in \Phii_\Aa(x),$ then $b$ is a limit point of a sequence $\{b^{(n)}\}_{n=1,2,\ldots}$ with $b^{(n)}\in \Phii_\Bb(x^{(n)},a^{(n)})$ for each $n=1,2,\ldots,$ that is, $\Phii_{\Bb}:\Gr(\Phii_{\Aa})\subset\Xx\times\Aa\mapsto S(\Bb)$ is $\Aa$-lower semi-continuous set-valued mapping.
\qed \end{proof}

The following two statements, which are not used in this paper, provide additional properties of   $\Aa$-lower semi-continuous set-valued mappings for the case, when $\Bb$ is a vector space. 
Let $\Bb$ be a vector space and $\Phii_\Bb,\Psii_\Bb:\Gr(\Phii_\Aa)\subset \Xx\times\Aa\mapsto S(\Bb)$ be set-valued mappings. Let us define for each $(x,a)\in \Gr(\Phii_\Aa)$
\[
\Phii_\Bb(x,a)+\Psii_\Bb(x,a):=\{b_0+b_1\,:\,b_1\in\Phii_\Bb(x,a),\, b_2\in\Psii_\Bb(x,a)
\}.
\]

\begin{lemma}\label{lem:linear}
Let $\Bb$ be a vector space and $\Phii_\Bb,\Psii_\Bb:\Gr(\Phii_\Aa)\subset \Xx\times\Aa\mapsto S(\Bb)$ be   $\Aa$-lower semi-continuous set-valued mappings. Then the set-valued mapping $
\Phii_\Bb+\Psii_\Bb:\Gr(\Phii_\Aa)\subset \Xx\times\Aa\mapsto S(\Bb)$ is   $\Aa$-lower semi-continuous.
\end{lemma}
\begin{proof} {\it of Lemma~\ref{lem:linear}}
Let $\{x^{(n)}\}_{n=1,2,\ldots}$ be a sequence with values in $\Xx$ that converges and its limit $x$ belongs to $\Xx.$ Assume that $a^{(n)}\in \Phii_\Aa(x^{(n)})$ for each $n=1,2,\ldots,$ and $b\in \Phii_\Bb(x,a)$ for some $a\in \Phii_\Aa(x).$ Let us prove that  $b$ is a limit point of a sequence $\{b^{(n)}\}_{n=1,2,\ldots}$ with $b^{(n)}\in \Phii_\Bb(x^{(n)},a^{(n)}),$ $n=1,2,\ldots.$ Indeed, since $\Phii_\Bb(x,a)=\Phii_\Bb(x,a)+\Psii_\Bb(x,a),$
 there exist $b_1\in \Phii_\Bb(x,a)$ and $b_2\in \Psii_\Bb(x,a)$  such that $b=b_0+b_1.$ The $\Aa$-lower semi-continuity of $\Phii_\Bb:\Gr(\Phii_\Aa)\subset \Xx\times\Aa\mapsto S(\Bb)$  and $\Psii_\Bb:\Gr(\Phii_\Aa)\subset \Xx\times\Aa\mapsto S(\Bb)$ imply that $b_i,$ $i=0,1,$ is a limit point of a sequence $\{b_i^{(n)}\}_{n=1,2,\ldots}$ with $b_1^{(n)}\in \Phii_\Bb(x^{(n)},a^{(n)})$ and $b_2^{(n)}\in \Psii_\Bb(x^{(n)},a^{(n)})$ Therefore, $b=b_0+b_1$ is a limit point of a sequence $\{b^{(n)}\}_{n=1,2,\ldots}$ with $b^{(n)}:=b_0^{(n)}+b_1^{(n)}\in \Phii_\Bb(x^{(n)},a^{(n)})+\Psii_\Bb(x^{(n)},a^{(n)}),$ 
$n=1,2,\ldots\ .$
Thus, the set-valued mapping $\Phii_\Bb^0+\Phii_\Bb^1:\Gr(\Phii_\Aa)\subset \Xx\times\Aa\mapsto S(\Bb)$ is   $\Aa$-lower semi-continuous.
\qed
\end{proof}

\begin{corollary}\label{cor:linear}
Let $\Bb$ be a vector space, $\Phii:\Xx\mapsto S(\Bb)$ be a lower semi-continuous set-valued mapping, $\Phii_{\Aa}:\Xx\mapsto \K(\Aa)$ be an upper semi-continuous set-valued mapping, and $\Psii_\Bb:\Gr(\Phii_\Aa)\subset \Xx\times\Aa\mapsto S(\Bb)$
be a lower semi-continuous set-valued mapping. Let us consider the set-valued mapping  $\Phii_\Bb:\Gr(\Phii_\Aa)\subset \Xx\times\Aa\mapsto S(\Bb)$ defined in \eqref{eqdefphphi}. Then the set-valued mapping $\Phii_\Bb+\Psii_\Bb:\Gr(\Phii_\Aa)\subset \Xx\times\Aa\mapsto S(\Bb)$
is   $\Aa$-lower semi-continuous.
\end{corollary}
\begin{proof}
According to Lemma~\ref{lem:unifAlsc}, the set-valued mappings $\Phii_\Bb, \Psii_\Bb:\Gr(\Phii_\Aa)\subset \Xx\times\Aa\mapsto S(\Bb)$ are    $\Aa$-lower semi-continuous. 
Therefore, Lemma~\ref{lem:linear} implies that their sum is    $\Aa$-lower semi-continuous.
\qed
\end{proof}

\begin{acknowledgements}
The authors thank William D. Sudderth for his valuable comments on von Neumann's and Sion's minimax theorems. The authors thank referees for
their insightful suggestions. 
\end{acknowledgements}

\end{document}